\documentclass[reqno]{gtpart}

\title{Symmetric functions in noncommuting variables in superspace}
\author[D. Arcis, C. Gonz\'alez and S. M\'arquez]{Diego Arcis\\Camilo Gonz\'alez\\Sebasti\'an M\'arquez}

\address{Departamento de Matem\'aticas, Universidad de La Serena, Cisternas 1200 -- 1700000 La Serena, Chile.\textcolor{white}{$\underbrace{1}$}\newline
Departamento de Matemáticas, Universidad de Concepción, Casilla 160--C -- 4030000 Concepción, Chile.
\textcolor{white}{$\underbrace{.}$}\newline
Departamento de Matem\'aticas, Universidad Aut\'onoma de Chile, Sede Talca, 4 Norte 95 -- 3460000 Talca, Chile.}

\email{diego.arcis@userena.cl}
\email{camgonzalezp@udec.cl}
\email{sebastian.marquez@uautonoma.cl}

 % symmetric functions

\usepackage[utf8]{inputenc}

\usepackage{
	accents,
	amsthm,
	amsmath,
	amsfonts,
	cancel, % \cancel 
    float,
    graphicx,
    mathrsfs,
    pb-diagram,
    shuffle,
    soul, % \st
    subcaption,
    tikz,
    tikz-cd
}

\usetikzlibrary{arrows}

\newcommand{\working}{1}
\newcommand{\fix}[2]{\ifnum\working=1{{\red #1} {\blue(#2)}}\else{#1}\fi}

\usepackage[bottom=3.5cm,left=2.8cm,right=2cm,top=3cm]{geometry}

\newtheorem{thm}{Theorem}[section]
\newtheorem{crl}[thm]{Corollary}
\newtheorem{lem}[thm]{Lemma}
\newtheorem{pro}[thm]{Proposition}

\theoremstyle{definition}
\newtheorem{dfn}[thm]{Definition}
\newtheorem{exm}[thm]{Example}

\newtheorem{rem}[thm]{Remark}

\def\blue{\color{blue}}\def\red{\color{red}}

\newcommand{\sNSym}{\mathrm{\bf sNSym}}
\newcommand{\NSym}{\mathrm{\bf NSym}}
\newcommand{\NCSym}{\mathrm{\bf NCSym}}

\newcommand{\sSym}{\mathrm{\bf sSym}}
\newcommand{\sQSym}{\mathrm{\bf sQSym}}
\newcommand{\QSym}{\mathrm{\bf QSym}}
\newcommand{\SSym}{\mathrm{\bf Sym}}
\newcommand{\sNCSym}{\mathrm{\bf sNCSym}}

\newcommand{\sPar}{\mathrm{sPar}}

\renewcommand{\J}{\mathcal{J}}

\renewcommand{\P}{\mathcal{P}}

\newcommand{\PBr}{\mathcal{P}\kern-1pt\mathcal{B}r}
\newcommand{\sPBr}{s\mathcal{P}\kern-1pt\mathcal{B}r}

\newcommand{\sC}{s\mathcal{C}}
\newcommand{\sP}{s\mathcal{P}}

\newcommand{\field}{\Q}
\newcommand{\sseries}[1][\!\!]{\field_{#1}^{\theta}\langle\kern-0.25em\langle x\rangle\kern-0.25em\rangle}
\newcommand{\pseries}[2][\!\!]{\field_{#1}^{\theta}\langle\kern-0.25em\langle x_1,\ldots,x_{#2}\rangle\kern-0.25em\rangle}
\newcommand{\cseries}[1][\!\!]{\field_{#1}^{\theta}[\kern-0.18em[x]\kern-0.18em]}

\newcommand{\N}{\mathbb{N}}

\newcommand{\df}{\operatorname{df}}

\newcommand{\inv}{\operatorname{inv}}

\newcommand{\std}{\operatorname{std}}

\newcommand{\perm}[1]{\sigma\!_{_{(#1)}}}
\newcommand{\sort}{\operatorname{sort}}

\newcommand{\te}{\tilde{e}}
\newcommand{\tp}{\tilde{p}}

\renewcommand{\th}{\tilde{h}}

\newcommand{\Sym}{\mathfrak{S}}

\newcommand{\ipro}[1]{\left\langle #1\right\rangle}

\newcommand{\fcirc}{\triangleright}

\numberwithin{equation}{section}

\newcommand{\funcs}[2]{#1\langle\kern-0.2em\langle #2\rangle\kern-0.2em\rangle}

\newcommand{\aspdf}{1}
\newcommand{\vcdraw}[1]{\vcenter{\hbox{#1}}}

\newcommand{\figele}{
	\ifnum\aspdf=1\[\left(\,\,\emptyset,\quad
		\vcdraw{\includegraphics{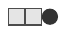}},\quad
		\vcdraw{\includegraphics{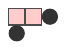}},\quad
		\vcdraw{\includegraphics{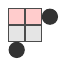}},\quad
		\vcdraw{\includegraphics{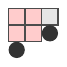}}\,\,\right)
	\]\else
		\tabpart{
			\snode\&\snode\&\dnode
		}
		\tabpart{
			\snode[red!20]\&\snode[red!20]\&\dnode\\
			\dnode
		}
		\tabpart{
			\snode[red!20]\&\snode[red!20]\&\dnode\\
			\snode\&\snode\\
			\dnode
		}
		\tabpart{
			\snode[red!20]\&\snode[red!20]\&\snode\\
			\snode[red!20]\&\snode[red!20]\&\dnode\\
			\dnode
		}
	\fi
}

\newcommand{\figten}{
	\ifnum\aspdf=1%\[
		\vcdraw{\includegraphics{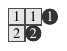}}
	\else
		\tabpart{
			\lnode{1}\&\lnode{1}\&\rnode{1}\\
			\lnode{2}\&\rnode{2}
		}
	\fi
}

\newcommand{\fignin}{
	\ifnum\aspdf=1%\[
		\vcdraw{\includegraphics{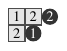}}
	\else
		\tabpart{
			\lnode{1}\&\lnode{2}\&\rnode{2}\\
			\lnode{2}\&\rnode{1}
		}
	\fi
}

\newcommand{\figeig}{
	\centering
	\ifnum\aspdf=1
		\includegraphics{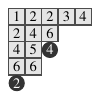}
	\else
		\tabpart{
			\lnode{1}\&\lnode{2}\&\lnode{2}\&\lnode{3}\&\lnode{4}\\
			\lnode{2}\&\lnode{4}\&\lnode{6}\\
			\lnode{4}\&\lnode{5}\&\rnode{4}\\
			\lnode{6}\&\lnode{6}\\
			\rnode{2}
		}
	\fi
}

\newcommand{\figsev}{
	\ifnum\aspdf=1\[\left(\,\,\emptyset,\quad
		\vcdraw{\includegraphics{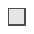}},\quad
		\vcdraw{\includegraphics{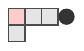}},\quad
		\vcdraw{\includegraphics{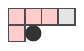}},\quad
		\vcdraw{\includegraphics{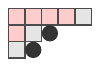}},\quad
		\vcdraw{\includegraphics{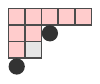}},\quad
		\vcdraw{\includegraphics{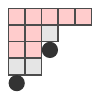}}\,\,\right)
	\]\else
		\tabpart{
			\snode
		}
		\tabpart{
			\snode[red!20]\&\snode\&\snode\&\dnode\\
			\snode
		}
		\tabpart{
			\snode[red!20]\&\snode[red!20]\&\snode[red!20]\&\snode\\
			\snode[red!20]\&\dnode
		}
		\tabpart{
			\snode[red!20]\&\snode[red!20]\&\snode[red!20]\&\snode[red!20]\&\snode\\
			\snode[red!20]\&\snode\&\dnode\\
			\snode\&\dnode
		}
		\tabpart{
			\snode[red!20]\&\snode[red!20]\&\snode[red!20]\&\snode[red!20]\&\snode[red!20]\\
			\snode[red!20]\&\snode[red!20]\&\dnode\\
			\snode[red!20]\&\snode\\
			\dnode
		}
		\tabpart{
			\snode[red!20]\&\snode[red!20]\&\snode[red!20]\&\snode[red!20]\&\snode[red!20]\\
			\snode[red!20]\&\snode[red!20]\&\snode\\
			\snode[red!20]\&\snode[red!20]\&\dnode\\
			\snode\&\snode\\
			\dnode
		}
	\fi
}

\newcommand{\figsix}{
	\ifnum\aspdf=1\[
		\left(
			\vcdraw{\includegraphics{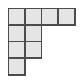}},
			\vcdraw{\includegraphics{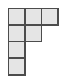}}
		\right)\kern+1.5cm
		\vcdraw{\includegraphics{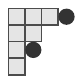}}\kern+1.5cm
		\vcdraw{\includegraphics{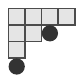}}
	\]\else
		\tabpart{
			\snode\&\snode\&\snode\&\snode\\
			\snode\&\snode\\
			\snode\&\snode\\
			\snode
		}
		\tabpart{
			\snode\&\snode\&\snode\\
			\snode\&\snode\\
			\snode\\
			\snode
		}
		\tabpart{
			\snode\&\snode\&\snode\&\dnode\\
			\snode\&\snode\\
			\snode\&\dnode\\
			\snode
		}
		\tabpart{
			\snode\&\snode\&\snode\&\snode\\
			\snode\&\snode\&\dnode\\
			\snode\\
			\dnode
		}
	\fi
}

\newcommand{\figfiv}{
	\centering
	\ifnum\aspdf=1
		\includegraphics{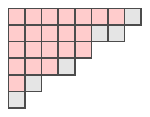}\qquad
		\includegraphics{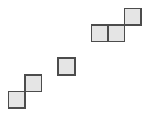}
	\else
		\tabpart{
			\snode[red!20]\&\snode[red!20]\&\snode[red!20]\&\snode[red!20]\&\snode[red!20]\&\snode[red!20]\&\snode[red!20]\&\snode\\
			\snode[red!20]\&\snode[red!20]\&\snode[red!20]\&\snode[red!20]\&\snode[red!20]\&\snode\&\snode\\
			\snode[red!20]\&\snode[red!20]\&\snode[red!20]\&\snode[red!20]\&\snode[red!20]\\
			\snode[red!20]\&\snode[red!20]\&\snode[red!20]\&\snode\\
			\snode[red!20]\&\snode\\
			\snode
		}

		\tabpart{
			\wnode\&\wnode\&\wnode\&\wnode\&\wnode\&\wnode\&\wnode\&\snode\\
			\wnode\&\wnode\&\wnode\&\wnode\&\wnode\&\snode\&\snode\\
			\wnode\&\wnode\&\wnode\&\wnode\\
			\wnode\&\wnode\&\wnode\&\snode\\
			\wnode\&\snode\\
			\snode
		}
	\fi
}

\newcommand{\figfou}{
	\centering
	\ifnum\aspdf=1
		\includegraphics{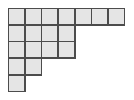}
	\else
		\tabpart{
			\snode\&\snode\&\snode\&\snode\&\snode\&\snode\&\snode\\
			\snode\&\snode\&\snode\&\snode\\
			\snode\&\snode\&\snode\&\snode\\
			\snode\&\snode\\
			\snode
		}
	\fi
}

\newcommand{\figthr}{
	\ifnum\aspdf=1\[
		\includegraphics[scale=0.8]{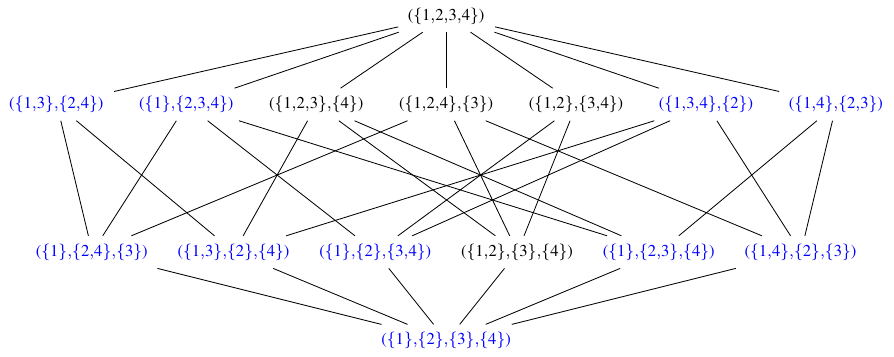}
	\]\else\[
		\begin{tikzpicture}
			\node(O)at(0,5.5){$_{(\{1,2,3,4\})}$};
			\node(H)at(-6.6,4){$_{\blue(\{1,3\},\{2,4\})}$};
			\node(I)at(-4.4,4){$_{\blue(\{1\},\{2,3,4\})}$};
			\node(J)at(-2.2,4){$_{(\{1,2,3\},\{4\})}$};
			\node(K)at(0,4){$_{(\{1,2,4\},\{3\})}$};
			\node(L)at(+2.2,4){$_{(\{1,2\},\{3,4\})}$};
			\node(M)at(+4.4,4){$_{\blue(\{1,3,4\},\{2\})}$};
			\node(N)at(+6.6,4){$_{\blue(\{1,4\},\{2,3\})}$};
			\node(B)at(-6.0,1.5){$_{\blue(\{1\},\{2,4\},\{3\})}$};
			\node(C)at(-3.6,1.5){$_{\blue(\{1,3\},\{2\},\{4\})}$};
			\node(D)at(-1.2,1.5){$_{\blue(\{1\},\{2\},\{3,4\})}$};
			\node(E)at(+1.2,1.5){$_{(\{1,2\},\{3\},\{4\})}$};
			\node(F)at(+3.6,1.5){$_{\blue(\{1\},\{2,3\},\{4\})}$};
			\node(G)at(+6.0,1.5){$_{\blue(\{1,4\},\{2\},\{3\})}$};
			\node(A)at(0,0){$_{\blue(\{1\},\{2\},\{3\},\{4\})}$};
			\draw(B)--(H);\draw(B)--(I);\draw(B)--(K);
			\draw(C)--(H);\draw(C)--(M);\draw(C)--(J);
			\draw(D)--(I);\draw(D)--(L);\draw(D)--(M);
			\draw(E)--(J);\draw(E)--(K);\draw(E)--(L);
			\draw(F)--(N);\draw(F)--(I);\draw(F)--(J);
			\draw(G)--(N);\draw(G)--(M);\draw(G)--(K);
			\draw(A)--(B);\draw(A)--(C);\draw(A)--(D);\draw(A)--(E);\draw(A)--(F);\draw(A)--(G);
			\draw(O)--(H);\draw(O)--(I);\draw(O)--(J);\draw(O)--(K);\draw(O)--(L);\draw(O)--(M);\draw(O)--(N);
		\end{tikzpicture}
	\]\fi
}

\newcommand{\figtwo}{
	\ifnum\aspdf=1\[
		\includegraphics[scale=0.8]{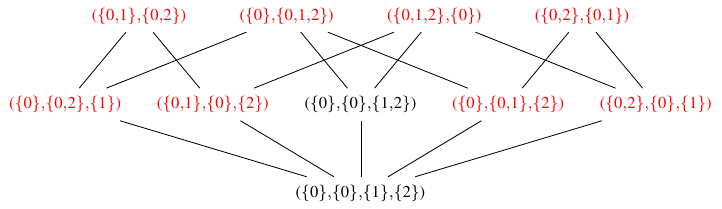}
	\]\else\[
		\begin{tikzpicture}
			\node(G)at(-3.75,3){$_{\red(\{0,1\},\{0,2\})}$};
			\node(H)at(-1.25,3){$_{\red(\{0\},\{0,1,2\})}$};
			\node(I)at(+1.25,3){$_{\red(\{0,1,2\},\{0\})}$};
			\node(J)at(+3.75,3){$_{\red(\{0,2\},\{0,1\})}$};
			\node(B)at(-5,1.5){$_{\red\red(\{0\},\{0,2\},\{1\})}$};
			\node(C)at(-2.5,1.5){$_{\red(\{0,1\},\{0\},\{2\})}$};
			\node(D)at(0,1.5){$_{(\{0\},\{0\},\{1,2\})}$};
			\node(E)at(+2.5,1.5){$_{\red\red(\{0\},\{0,1\},\{2\})}$};
			\node(F)at(+5,1.5){$_{\red(\{0,2\},\{0\},\{1\})}$};
			\node(A)at(0,0){$_{(\{0\},\{0\},\{1\},\{2\})}$};
			\draw(A)--(B);\draw(A)--(C);\draw(A)--(D);\draw(A)--(E);\draw(A)--(F);
			\draw(G)--(B);\draw(G)--(C);
			\draw(H)--(B);\draw(H)--(D);\draw(H)--(E);
			\draw(I)--(C);\draw(I)--(D);\draw(I)--(F);
			\draw(J)--(E);\draw(J)--(F);
		\end{tikzpicture}
	\]\fi
}

\newcommand{\figone}{
	\ifnum\aspdf=1\[
		\includegraphics[scale=0.8]{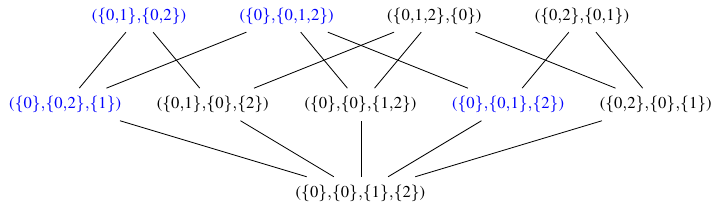}
	\]\else\[
		\begin{tikzpicture}
			\node(G)at(-3.75,3){$_{\blue(\{0,1\},\{0,2\})}$};
			\node(H)at(-1.25,3){$_{\blue(\{0\},\{0,1,2\})}$};
			\node(I)at(+1.25,3){$_{(\{0,1,2\},\{0\})}$};
			\node(J)at(+3.75,3){$_{(\{0,2\},\{0,1\})}$};
			\node(B)at(-5,1.5){$_{\blue(\{0\},\{0,2\},\{1\})}$};
			\node(C)at(-2.5,1.5){$_{(\{0,1\},\{0\},\{2\})}$};
			\node(D)at(0,1.5){$_{(\{0\},\{0\},\{1,2\})}$};
			\node(E)at(+2.5,1.5){$_{\blue(\{0\},\{0,1\},\{2\})}$};
			\node(F)at(+5,1.5){$_{(\{0,2\},\{0\},\{1\})}$};
			\node(A)at(0,0){$_{(\{0\},\{0\},\{1\},\{2\})}$};
			\draw(A)--(B);\draw(A)--(C);\draw(A)--(D);\draw(A)--(E);\draw(A)--(F);
			\draw(G)--(B);\draw(G)--(C);
			\draw(H)--(B);\draw(H)--(D);\draw(H)--(E);
			\draw(I)--(C);\draw(I)--(D);\draw(I)--(F);
			\draw(J)--(E);\draw(J)--(F);
		\end{tikzpicture}
	\]\fi
}

\begin{document}

\begin{abstract}
In 2004, Rosas and Sagan developed the theory of symmetric functions in noncommuting variables, achieving results analogous to classical symmetric functions. On the other hand, the same year, Desrosiers, Lapointe and Mathieu introduced the theory of symmetric functions in superspace, involving both commuting and anticommuting variables, extending the classic theory. Here, we introduce symmetric functions in noncommuting variables in superspace. We extend the classical symmetric functions in noncommuting variables to superspace: monomial, power sum, elementary and complete homogeneous functions. These functions generalize both those studied by Rosas and Sagan and those studied by Desrosiers, Lapointe, and Mathieu. Additionally, we define Schur--type functions in noncommuting variables in superspace.
\end{abstract}

\maketitle

%\begin{center}\begin{minipage}{13cm}\tableofcontents\end{minipage}\end{center}

\section{Introduction}

Symmetric functions play a significant role in mathematics, particularly in areas such as algebraic combinatorics, representation theory, and Hopf algebras, among others \cite{Mac99,St99}. They have been widely studied and are important not only in mathematics but also in connection with integrable models in physics \cite{MiJiDa00}. Specifically, symmetric functions turn out to be eigenfunctions of the Hamiltonian of the Calogero--Moser--Sutherland model \cite{Ruij99,VaDi00}. Roughly speaking, symmetric functions are formal power series in infinitely many commuting variables, remaining invariant under any permutation of their variables. Symmetric functions form an algebra with bases indexed by partitions of integers. Some of these bases include the monomial symmetric functions, power sum symmetric functions, elementary symmetric functions, and complete homogeneous symmetric functions.

An important family of symmetric functions that form a basis are the so-called Schur functions, appearing as limits of characters of the general linear group and associated with the irreducible characters of the symmetric group via the Schur--Weyl duality. Schur functions can be defined through various methods: by using semistandard Young tableaux, via the Jacobi--Trudy identity, or directly through a quotient involving Vandermonde-type determinants \cite[Chapter 7]{St99}.

Several generalizations of the algebra of symmetric functions $\SSym$ have been studied, such as the algebra of quasisymmetric functions $\QSym$ \cite{Ges84} and its dual Hopf algebra of noncommuting symmetric functions $\NSym$ \cite{GKLLRT95}, among others \cite{PoRe95,DuHiTh02,GrRe14}. Quasisymmetric functions can be regarded as a refinement of symmetric functions in the sense that they remain invariant under a weaker form of symmetry.

Another generalization of symmetric functions involves noncommuting variables, introduced by Wolf in 1936 \cite{Wol36}. The algebra of symmetric functions in noncommuting variables $\NCSym$ was studied by Rosas and Sagan in \cite{RoSa06}, establishing results similar to the commuting case. In particular, they defined bases analogous to the monomial, power, elementary, and complete homogeneous functions, indexed by set partitions. They showed relations between these bases and studied their projection to the algebra of symmetric functions in commuting variables. Also, by using combinatorics of tableaux, they introduced a family of Schur--type functions in noncommuting variables, satisfying properties similar to the classical Schur functions; however, this family does not constitute a basis of $\NCSym$. Using a noncommuting version of the Jacobi--Trudi determinant, a new family of Schur--type functions was introduced in \cite{AlXiWi22}, refining the one in \cite{RoSa06} and constituting a basis of $\NCSym$. Several results related to this algebra are known, such as the isomorphism with the algebra of rook placements \cite{CaSa11} and the connection with supercharacter theory \cite{AgEtAl12,Wol36}. Some results related to the Hopf structure of $\NCSym$ were obtained in \cite{BeReRoZa08,BeHoRoZa06,BeZa09}.

Given the connection of symmetric functions with integrable models, and motivated by a supersymmetric generalization of the Calogero--Moser--Sutherland model \cite{DeLaMa01,DeLaMaProc04}, it was developed in \cite{DeLaMa03,DeLaMa04} a new family of symmetric functions called \emph{symmetric functions in superspace}, involving commuting variables $x=(x_1,x_2,\ldots)$ together with anticommuting variables $\theta=(\theta_1,\theta_2,\ldots)$, in which the symmetric group acts simultaneously. These functions have been widely studied in recent years, extending results of classic symmetric functions. Classic bases such as monomial, power sum, elementary, and complete homogeneous symmetric functions were extended to superspace in \cite{DeLaMa06}, indexed by \emph{superpartitions}. Several definitions of Schur--type functions in superspace were given in \cite{JoLa17}. On the other hand, it was proved in \cite{FiLaPi19} that $\QSym$ and $\NSym$ can be generalized to superspace as the Hopf algebras of quasisymmetric functions in superspace $\sQSym$ and of noncommuting symmetric functions in superspace $\sNSym$, respectively. Some results related to the Hopf structure of $\sNSym$ were obtained in \cite{ArGoMa24}.

A natural problem is to consider symmetric functions involving noncommuting and anticommuting variables, which generalize both symmetric functions in noncommuting variables and symmetric functions in superspace. Here, we address this problem. Specifically, we extend the classical bases and the Schur--type functions in noncommuting variables to superspace, which project naturally onto their commuting counterparts in superspace. The paper is organized as follows.

In Section 2, we introduce two new classes of combinatorial objects called \emph{partial set supercompositions} and \emph{set superpartitions}, which generalize classic set partitions and are closely related to superpartitions. These objects will be used to study the bases of symmetric functions in noncommuting variables in superspace. We define \emph{strongly coarsening} and \emph{coarsening} relations on these combinatorial objects, giving them structures of partially ordered sets. Also, we describe actions of the symmetric group on these objects.

In Section 3, we introduce the algebra of symmetric functions in noncommuting variables in superspace $\sNCSym$. We define monomial symmetric functions in superspace and show that these ones form a basis of $\sNCSym$. Also, we provide combinatorial formulas for the product of these functions. Other bases of $\sNCSym$ are studied in Section 4. In particular, we define power sum, elementary, and complete homogeneous symmetric functions in superspace. We also show the relations between these bases.

Section 5 is devoted to studying how symmetric functions in noncommuting variables in superspace project to the algebra of functions in commuting variables in superspace $\sSym$. First, we recall classic symmetric functions in commuting variables in superspace. Then, we show formulas that project our bases onto their commuting counterparts, and we prove that there is an isometry from $\sNCSym$ to $\sSym$.

In Section 6, we define Schur--type functions in noncommuting variables in superspace in the sense of \cite{RoSa06}, using super semistandard Young tableaux introduced in \cite{JoLa17}. We express these functions in terms of the monomial basis and show their projection onto Schur--type functions in commuting variables in superspace, also given in \cite{JoLa17}.

Here, $\N$ denotes the set of positive integers, and $\N_0:=\N\cup\{0\}$.

Given integers $m,n$, denote by $[m,n]$ the interval $\{m,\ldots,n\}$, and $[m,n]_0:=[m,n]\cup\{0\}$. In particular, if $m=1$ we will denote them, respectively, by $[n]$ and $[n]_0$ instead. Note that $[0]$ is empty and $[0]_0=\{0\}$.

For every integer $n\geq1$, we denote permutations $\sigma:[n]\to[n]$ as words $\sigma(1)\cdots\sigma(n)$ in $[n]$, and we denote by $\inv(\sigma)$ the \emph{number of inversions} of $\sigma$, that is, the minimal number of commutations needed to get $12\cdots n$ from $\sigma(1)\cdots\sigma(n)$. The \emph{symmetric group} of permutations of $[n]$, will be denoted by $\Sym_n$, and we set ${\displaystyle\Sym=\bigsqcup_{n\geq1}\Sym_n}$.

\section{Set superpartitions}

Here we introduce a new class of combinatorial objects, which will index the bases of the algebra of symmetric functions in superspace. These ones can be regarded as a extension of the classic set partitions. Note that set superpartitions introduced here are different of the ones defined in \cite{RhTi22}.

\subsection{Preliminaries}

Recall that a \emph{set partition} of a set $X$ is a collection $I$ of nonempty subsets $I_1,\ldots,I_k$ of $X$, called \emph{blocks}, such that $I_1\cup\cdots\cup I_k=X$ and $I_i\cap I_j=\emptyset$ for all $i<j$. If $X$ is a set of integers, it is usual to denote $I=\{I_1,\ldots,I_k\}$ by $I=(I_{a_1},\ldots,I_{a_k})$ instead, where $\min(I_{a_i})<\min(I_{a_j})$ whenever $i<j$. The collection of set partitions of $X$ is denoted by $\P(X)$. For $I,J\in\P(X)$, write $I\preceq J$ if each block of $J$ is a union of blocks of $I$. It is well known that $\preceq$ is a partial order that gives to $\P(X)$ a structure of lattice with \emph{meet operation} $\wedge$ and \emph{join operation} $\vee$. For instance, for $X=[4]$ the lattice structure of $\P(X)$ is given as in Figure \ref{026}.
\begin{figure}[H]
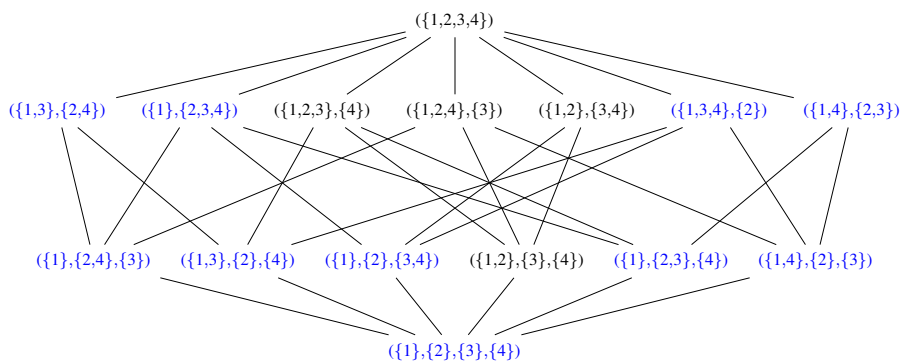

\figthr
\caption{Images of partial set supercompositions of bidegree $(2,2)$ in the lattice of set partitions of $[4]$.}
\label{026}
\end{figure}
For every nonnegative integer $n$, the set $\P([n])$ is commonly denoted by $\P_n$ instead. It is well known that the number of set partitions of $[n]$ into $1\leq k\leq n$ blocks is the \emph{Stirling number} (of the second kind) $S(n,k)$ \cite[A008277]{OEIS}, and that the number of all set partitions of $[n]$ is the $n$th Bell number $b_n$ \cite[A000110]{OEIS}.

\subsection{Partial set supercompositions}

In what follows, for every subset $A$ of $\N_0$, we denote $A\backslash\{0\}$ by $A^+$ instead. Also, for subsets $A,B$ of $\N_0$, we set $A\oplus B$ as empty if $0\in A\cap B$ and $A\oplus B=A\cup B$ otherwise. This operation is not associative in general, indeed $(\{1\}\oplus\{0,1\})\oplus\{0,1\}=\{\}$ and $\{1\}\oplus(\{0,1\}\oplus\{0,1\})=\{1\}$. However, if either none or at most one of the sets $A,B,C$ contains zero, then $(A\oplus B)\oplus C=A\oplus(B\oplus C)$.

A \emph{partial set supercomposition} $K$, of \emph{degree} $n=:\deg(K)$ and \emph{fermionic degree} $m=:\df(K)\leq n+1$, is a tuple $(K_1,\ldots,K_q)$ of nonempty subsets of $[n]_0$, called \emph{blocks}, satisfying the following conditions:\begin{enumerate}
\item $(K_1\cup\cdots\cup K_q)^+=[n]$ and $K_i\cap K_j\subseteq\{0\}$ for all $i\neq j$.
\item $0\in K_i$ for all $i\in[m]$ and $\min(K_i)<\min(K_j)$ for all $m\leq i<j\leq q$.
\end{enumerate}
Each block of $K$ containing $0$ is called a \emph{fermionic block}. We say that $(n,m)$ is the \emph{bidegree} of $K$. The number of blocks of $K$ is called the \emph{length} of it and is denoted by $\ell(K)$. The collection $K^+$ formed by the nonempty sets $K_i^+$ with $i\in[q]$ is a set partition of $[n]$ which we call the \emph{underlying set partition} of $K$. For instance, below a partial set supercomposition of bidegree $(7,3)$:\begin{equation}\label{001}K=(\{0,4\},\{0,2,7\},\{0\},\{1,5,6\},\{3\}).\end{equation}Observe that, by definition, partial set supercompositions may contain many blocks $\{0\}$. For instance, the following is a partial set supercomposition of bidegree $(n,m)$:\[\hat{0}_{n,m}=(\underbrace{\{0\},\ldots,\{0\}}_{m\text{--times}},\{1\},\ldots,\{n\}).\]The collection of partial set supercompositions of degree $n$ will be denoted by $\sC_n$, and the subcollection of it formed by the ones of fermionic degree $m$ will be denoted by $\sC_{n,m}$.

For every $K=(K_1,\ldots,K_q)\in\sC_{n,m}$ and every pair of permutations $(\sigma,\delta)\in\Sym_m\times\Sym_n$, we define\[\sigma\fcirc K=(K_{\sigma(1)},\ldots,K_{\sigma(m)},K_{m+1},\ldots,K_q)\qquad\text{and}\qquad\delta\circ 
K=(\delta(K_1),\ldots,\delta(K_q))
%K=(K_1,\ldots,K_m,\delta(K_{m+1}),\ldots,\delta(K_q))
.\]Observe that the operators $\fcirc$ and $\circ$ define group actions of $\Sym_n$ and $\Sym_m$, respectively, on the collection of partial set supercompositions of bidegree $(n,m)$. Thus $\sigma\fcirc(\sigma'\fcirc K)=(\sigma\sigma')\fcirc K$ and $\delta\circ(\delta'\circ K)=(\delta\delta')\circ K$.

\subsection{Coarsening and strongly coarsening}

Here we stablish two partial orders on partial super set compositions.

Given $K,L\in\sC_{n,m}$, we say that $L$ is \emph{strongly coarser} than $K$, or that $K$ is \emph{strongly finer} than $L$, denoted by $K\sqsubseteq L$, if the following conditions hold:
\begin{enumerate}
\item The $i$th fermionic block of $L$ is the $i$th fermionic block of $K$ joined with some possible nonfermionic blocks of $K$.
%\item The $i$th fermionic block of $L$ is either the $i$th fermionic block of $K$ or this one joined with some nonfermionic blocks of $K$.
\item Each nonfermionic block of $L$ is a union of nonfermionic blocks of $K$.
\end{enumerate}In this case $K$ is called a \emph{strongly refinement} of $L$. The \emph{strongly coarsening} relation $\sqsubseteq$ is a partial order on $\sC_{n,m}$ that gives to it a structure of complete meet--semilattice with operation $\sqcap$ and infimum $\hat{0}_{n,m}$. See Figure \ref{024}. Observe that partial set supercomposition whose blocks are all fermionic are maximal with respect to $\sqsubseteq$.
\begin{figure}[H]
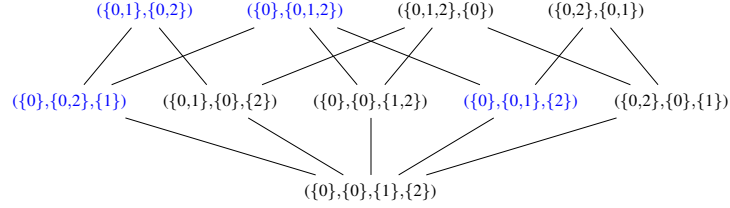

\figone
\caption{Set superpartitions in the meet--semilattice of partial set supercompositions of bidegree $(2,2)$.}
\label{024}
\end{figure}

For $r\in\Z$, the \emph{$r$-shift} of $A\subseteq\N_0$ is the set $A[r]$ obtained by adding $r$ to each nonzero element of $A$. The \emph{$r$-shift} of a partial set supercomposition $K=(K_1,\ldots,K_q)$ is the collection $K[r]=(K_1[r],\ldots,K_q[r])$. For instance, if $K=(\{0\},\{0,3,5\},\{1,2,4\})$ then $K[2]=(\{0\},\{0,5,7\},\{3,4,6\})$.

The \emph{standardization} of $K\in\sC_{n,m}$ is the set partition $\std(K)\in\P_{n+m}$ obtained from $K[m]$ by renumbering the zeros from $1$ to $m$ with respect to the sorting of the blocks of $K$. This operator defines an injective order preserving map of $(\sC_{n,m},\sqsubseteq)$ into $(\P_{n+m},\preceq)$. Observe that $I\in\P_{n+m}$ belongs to $\std(\sC_{n,m})$ if and only if the elements of $[m]$ belong to different blocks of $I$. For instance $\std(\{0,1\},\{0\},\{2\})=(\{1,3\},\{2\},\{4\})$.
\begin{pro}\label{027}
The image of $\sC_{n,m}$ under the standardization map is a convex subposet of $\P_{n+m}$. In consequence, due to \cite[Lemma, p. 359]{Ro64}, for every partial set supercomposition $K=(K_1,\ldots,K_q)$, we have\[[\hat{0}_{n,m},K]\simeq\prod_{i=1}^q\P_{|K_i|}.\]
\end{pro}
\begin{proof}
Let $K,L\in\sC_{n,m}$, and let $I\in\P_{n+m}$ such that $\std(K)\preceq I\preceq\std(L)$. Since $L$ belongs to $\sC_{n,m}$ and $I\preceq\std(L)$, then every pair of elements of $[m]$ belong to different blocks of $I$. This implies that $I=\std(N)$ for some $N\in\sC_{n,m}$. Therefore $\std(\sC_{n,m})$ is a convex subposet of $\P_{n+m}$. See Figure \ref{026}.
\end{proof}

For $K,L\in\sC_n$, we say that $L$ is \emph{coarser} than $K$, or that $K$ is \emph{finer} than $L$, denoted by $K\preceq L$, if each block of $L$ is obtained by $\oplus$--operating some blocks of $K$. In this case $K$ is called a \emph{refinement} of $L$. Observe that as $L$ contains no empty blocks, then it cannot be obtained by $\oplus$--operating more than one fermionic block of $K$. With this \emph{coarsening relation}, the collection $\sC_n$ becomes a partially ordered set. Note that every partial set supercomposition whose blocks are all fermionic is maximal with respect to $\preceq$.

\subsection{Set superpartitions}

We say that $I=(I_1,\ldots,I_k)\in\sC_n$ is a \emph{set superpartition} if $I_i\neq I_j$ and $\min(I_i\backslash I_j)<\min(I_j\backslash I_i)$ for all $i,j\in[k]$ with $i<j$, where $\min(\emptyset):=0$. The collection of set superpartitions of degree $n$ will be denoted by $\sP_n$, and the subcollection of it formed by the ones of fermionic degree $m$ will be denoted by $\sP_{n,m}$.

Observe that every set superpartition that does not contain the block $\{0\}$ is obtained by inserting, possibly, a zero in some blocks of a usual set partition. This implies that, in $\sP_n$, there are $sb_n^*$ \cite[A001861]{OEIS} of these set superpartitions. Note that $sb_n^*$ is obtained by multiplying every addend of $b_n$ by $2^k$. Thus $|\sP_n|=2sb_n^*=:sb_n$.

We say that $I\in\sP_n$ is \emph{convex} if each of its blocks is an interval, that is, for each $B\in I$ there are $a,b\in[n]$ such that, either $B=[a,b]$ or $B=[a,b]_0$. For instance, the following is convex of bidegree $(7,3)$:\[I=(\{0\},\{0,1,2\},\{0,3\},\{4,5,7\},\{6\}).\]Observe that every $I\in\sP_{n,m}$ can be written by means of the $\circ$--action of a convex set superpartition $J$, that is, $I=\delta\circ J$. For instance, if $I=(\{0\},\{0,4\},\{0,6\},\{1,3,5\},\{2\})$, then\[I=\delta\circ\underbrace{(\{0\},\{0,1\},\{0,2\},\{3,4,5\},\{6\})}_J,\quad\text{where}\quad\delta=461352.\]

For every partial set supercomposition $K$, we will denote by $\bar{K}$ the unique set superpartition obtained from $K$ by removing its possible repeated blocks and so sorting the fermionic blocks. For instance:\[K=(\{0,3\},\{0\},\{0,2\},\{0\},\{1,5\},\{4\})\quad\to\quad\bar{K}=(\{0\},\{0,2\},\{0,3\},\{1,5\},\{4\}).\]We say that $K\in\sC_{n,m}$ is \emph{nontrivial} if it has at most one block $\{0\}$, that is, there are $\sigma\in\Sym_m$ and $I\in\sP_{n,m}$ such that $K=\sigma\fcirc I$. See Figure \ref{025}.\begin{figure}[H]
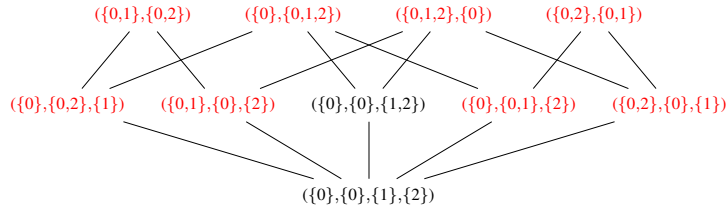

\figtwo
\caption{Nontrivial elements in the meet--semilattice of partial set supercompositions of bidegree $(2,2)$.}
\label{025}
\end{figure}
Observe that as nontrivial partial set supercompositions are obtained by resorting the fermionic blocks of set superpartitions, then $K\sqsubseteq L$ implies that $\bar{K}\preceq\bar{L}$.

In what follows, for every $A\subseteq[n]_0$ and every collection $K=(K_1,\ldots,K_q)$ of subsets of $[n]_0$, we set\[K\backslash A=(K_1\backslash A,\ldots,K_q\backslash A)\,\backslash\,\emptyset.\]
\begin{pro}\label{022}
Given $K=(K_1,\ldots,K_h)$ and $L=(L_1,\ldots,L_k)$ in $\sC_{n,m}$, we have $K\sqcap L=(N_1,\ldots,N_q)$, where $N_i=K_i\cap L_i$ for all $i\in[m]$, and $(N_{m+1},\ldots,N_q)=(K^+\backslash A)\wedge(L^+\backslash A)$ with $A=N_1\cup\cdots\cup N_m$.
\end{pro}
\begin{proof}
By definition, we have $N\sqsubseteq K$ and $N\sqsubseteq L$, where $N=(N_1,\ldots,N_q)$ as on the statement. Thus $N\sqsubseteq K\sqcap L$. Assume that $N\sqsubset K\sqcap L$, then $K\sqcap L$ contains a block $B=N_i\cup N_j$ for some $i,j\in[q]$ with $i<j$. Now, if $N_i$ is fermionic, then $B\subset K_i$ and $B\subset L_i$, which is a contradiction because $N_i=K_i\cap L_i$. On the other hand, if $N_i$ is not fermionic, then $B\cap A=\emptyset$ and $B\subseteq K_a\cap H_b$ for some $a,b\geq m$. This implies that $B$ is contained in some block of $(K^+\backslash A)\wedge(L^+\backslash A)$, which also is a contradiction. Therefore $N=K\sqcap L$.
\end{proof}

\section{Functions in superspace}\label{067}

In this section, we define the algebra of symmetric functions in noncommuting variables in superspace. We introduce monomial symmetric functions in superspace and show that these ones form a basis of it. Also, we give combinatorial formulas for the product of these functions.

Let $\sseries$ be the $\field$--algebra of formal power series of bounded degree presented by infinite generators $x_1,x_2,\ldots$ and $\theta_1,\theta_2,\ldots$, called \emph{variables}, subject to the relations $x_i\theta_j=\theta_jx_i$ and $\theta_i\theta_j=-\theta_j\theta_i$ for all $i,j\in\N$. Every element of $\sseries$ will be called a \emph{function (in noncommuting variables) in superspace}. Generators in $x:=(x_1,x_2,\ldots)$ will be called \emph{$x$--variables} and generators in $\theta:=(\theta_1,\theta_2,\ldots)$ will be called \emph{fermionic variables} or simply \emph{$\theta$--variables}. Observe that $\theta_i^2=0$ for all $i\in\N$.

Due to the relations that define $\sseries$, every monomial occurring in a function in superspace can be uniquely written as a product $q\theta_{a_1}\cdots\theta_{a_r}x_{b_1}\cdots x_{b_s}$ up the relations that define $\sseries$, where $q\in\field$ and $a_i,b_j\in\N$. For every nonnegative integer $n$ we will denote by $\sseries[n]$ the subspace formed by the functions in superspace whose monomials have exactly $n$ occurrences of $x$--variables. Similarly, for every nonnegative integer $m$ we denote by $\sseries[n,m]$ the subspace of functions in $\sseries[n]$ whose monomials have exactly $m$ occurrences of $\theta$--variables. Hence, we obtain the following decompositions:\[\sseries=\bigoplus_{n\geq0}\sseries[n]=\bigoplus_{n,m\geq0}\!\sseries[n,m].\]We call $n$ the \emph{degree} and $m$ the \emph{fermionic degree} of the elements of $\sseries[n,m]$.

\subsection{Symmetric functions in superspace}

For every $n\geq1$ the group $\Sym_n$ acts linearly on $\sseries$ via permutation of variables, that is, $s(x_i)=x_{s(i)}$ and $s(\theta_i)=\theta_{s(i)}$ for all $(s,i)\in\Sym_n\times\N$, where $s(i):=i$ whenever $i>n$. An element of $\sseries$ is called \emph{symmetric} if it is invariant under the action of $\Sym_n$ for all $n\in\N$. We will denote by $\sNCSym$ the subalgebra of $\sseries$ formed by the symmetric functions in superspace. Thus, we have\begin{equation}\label{000}\sNCSym=\bigoplus_{n\geq0}\sNCSym_n=\bigoplus_{n,m\geq0}\sNCSym_{n,m},\end{equation}where $\sNCSym_n=\sNCSym\cap\sseries[n]$ and $\sNCSym_{n,m}=\sNCSym\cap\sseries[n,m]$ for all $n,m\in\N_0$.

Clearly, the action of the symmetric group on $\sNCSym$ keep the number of occurrences of variables in monomials, that is, $\sigma(\sNCSym_{n,m})\subseteq\sNCSym_{n,m}$ for all permutation $\sigma$ and all $n,m\in\N_0$.

A monomial $q\theta_{a_1}\cdots\theta_{a_m}x_{b_1}\cdots x_{b_n}$ is said to be \emph{null--symmetric} if $\{a_1,\ldots,a_m\}\backslash\{b_1,\ldots,b_n\}$ has more than one element. For instance, the monomial $\theta_1\theta_2x_2$ is non null--symmetric, however $\theta_1\theta_3x_2$ is null--symmetric.

\begin{pro}\label{003}
There are no symmetric functions in superspace involving null--symmetric  monomials. In consequence, $\sNCSym_{n,m}=\{0\}$ for all $n,m\in\N_0$ with $m>n+1$. Thus,\[\sNCSym_n=\bigoplus_{m=0}^{n+1}\sNCSym_{n,m}.\]
\end{pro}
\begin{proof}
Let $f$ be a symmetric function in superspace, and let $u=q\theta_{a_1}\cdots\theta_{a_m}x_{b_1}\cdots x_{b_n}$ be a monomial occurring in it. If $u$ is null--symmetric, there are $i,j\in[m]$ such that $a_i,a_j\not\in\{b_1,\ldots,b_n\}$. Consider $\sigma$ be the transposition exchanging $a_i$ with $a_j$. As $f$ is symmetric, $\sigma(u)=-u$ also occurs in $f$. This holds for all null--symmetric monomial occurring in $f$. The equality $\sNCSym_{n,m}=\{0\}$ when $m>n+1$ is a consequence of the fact that every monomial with $n$ $x$--variables and $m>n+1$ $\theta$--variables is null--symmetric. 
\end{proof}

In order to study the bases of $\sNCSym$, we will consider other actions of the symmetric group on $\sseries$ in which $\Sym_n$ acts independently on the $x$--variables and on the $\theta$--variables. More specifically, given a monomial $u=q\theta_{a_1}\cdots\theta_{a_m}x_{b_1}\cdots x_{b_n}$ and permutations $(\sigma,\delta)\in\Sym_m\times\Sym_n$, we define\begin{equation}\label{055}\sigma\fcirc u=q\theta_{a_{\sigma(1)}}\cdots\theta_{a_{\sigma(m)}}x_{b_1}\cdots x_{b_n}\qquad\text{and}\qquad\delta\circ u=q\theta_{a_1}\cdots\theta_{a_m}x_{b_{\delta^{-1}(1)}}\cdots x_{b_{\delta^{-1}(n)}}.\end{equation}Observe that, as $\theta$--variables anticommute, we have $\sigma\fcirc u=(-1)^{\inv(\sigma)}u$.

\subsection{Monomial functions in superspace}

Here we introduce the basis of monomial functions in superspace.

\begin{dfn}\label{005}
The \emph{monomial function} of a partial set supercomposition $K=(K_1,\ldots,K_q)$ of bidegree $(n,m)$ is the function in superspace, defined as follows:\[m_K\,\,\,\,=\sum_{(a_1,\ldots,a_m,b_1,\ldots,b_n)}\theta_{a_1}\cdots\theta_{a_m}x_{b_1}\cdots x_{b_n},\]where\begin{enumerate}
\item For each $i\in[m]$, $a_i=b_j$ if and only if $j\in K_i^+$.\label{010}
\item For each $i,j\in[n]$, $b_i=b_j$ if and only if $i,j$ belong to the same block of $K$.\label{002}
\end{enumerate}
We will denote by $M(K)$ the set of all monomials satisfying conditions \eqref{010} and \eqref{002} above.
\end{dfn}
For instance, if $K=(\{0,2\},\{0,1,3\},\{4\})$ in $\sC_{4,2}$, we have
\[m_K=\theta_2\theta_1x_1x_2x_1x_3+\theta_3\theta_1x_1x_3x_1x_2+\theta_1\theta_2x_2x_1x_2x_3+\theta_3\theta_2x_2x_3x_2x_1+\theta_1\theta_3x_3x_1x_3x_2+\theta_2\theta_3x_3x_2x_3x_1+\cdots\]

\begin{pro}\label{004}
For every partial set supercomposition $K$, the function $m_K$ is symmetric. Moreover, for every monomial $u\in M(K)$, we have\[m_K=\sum_{\sigma\in\Sym}\sigma(u).\]
\end{pro}
\begin{proof}
It is a direct consequence of the fact that $\sigma(M(K))=M(K)$ for all $\sigma\in\Sym$.
\end{proof}

Observe that if $K\in\sC_{n,m}$ contains at least two blocks $\{0\}$, then $m_K$ involves only null--symmetric monomials, thus, by Proposition \ref{003}, $m_K=0$. Now, recall that if $K$ is nontrivial, then $K=\sigma\fcirc I$ for some $\sigma\in\Sym_m$ and $I\in\sP_{n,m}$. Thus, the monomial function $m_K$ can be characterized by applying the $\fcirc$--action of $\sigma$, that is\[\sigma\fcirc m_I\,\,\,\,=\sum_{(a_1,\ldots,a_m,b_1,\ldots,b_n)}\theta_{a_{\sigma(1)}}\cdots\theta_{a_{\sigma(m)}}x_{b_1}\cdots x_{b_n}=m_{\sigma\,\fcirc\,I}=m_K.\]

\begin{pro}\label{006}
We have
\begin{enumerate}
\item Every non null--symmetric monomial belongs to a unique $M(\sigma\fcirc I)$.\label{012}
\item The set $\{m_{\sigma\,\fcirc\,I}\mid\sigma\in\Sym_m,\,I\in\sP_{n,m}\}$ spans $\sNCSym_{n,m}$.\label{011}
\item For every $I\in\sP_{n,m}$ and $\sigma\in\Sym_m$, $m_{\sigma\,\fcirc\,I}=(-1)^{\inv(\sigma)}m_I$.\label{007}
\item The set $\{m_I\mid I\in\sP_{n,m}\}$ is a basis of $\sNCSym_{n,m}$.\label{068}
\end{enumerate}
\end{pro}
\begin{proof}
Let $u=\theta_{a_1}\cdots\theta_{a_m}x_{b_1}\cdots x_{b_n}$ be a non null--symmetric monomial, let $I^+$ be the unique set partition of $[n]$ in which $i,j$ belong to the same block if and only if $b_i=b_j$, and let $I$ to be the unique set superpartition of $[n]$ obtained by adding $0$ to each block of $I$ containing $i$ such that $b_i\in\{a_1,\ldots,a_m\}$. As $\theta$--variables are not necessarily sorted as the blocks of $I$, then there is a unique permutation $\sigma\in\Sym_m$ such that $u\in M(\sigma\fcirc I)$.

For every $f\in\sNCSym_{n,m}$, Proposition \ref{003} and Proposition \ref{004} imply that if $qu$ occurs in $f$ for some $q\in\field$ and $u\in M(\sigma\fcirc I)$ with $(\sigma,\,I)\in\Sym_m\times\sP_{n,m}$ as in \eqref{012}, then $qm_I$ occurs in $f$ as well. As it holds for every monomial occurring in $f$ and the set $\Sym_m\times\sP_{n,m}$ is finite, then $f$ is a linear combination of monomial functions in $\{m_{\sigma\,\fcirc\,I}\mid\sigma\in\Sym_m,\,I\in\sP_{n,m}\}$. Thus \eqref{011} holds.

To show \eqref{007}, it is enough to note that, by Definition \ref{005}, $m_I$ is obtained by permuting by $\sigma^{-1}$ the $\theta$--variables of every monomial of $M(\sigma\fcirc I)$, that is, $m_I=(-1)^{\inv(\sigma)}m_{\sigma\,\fcirc\,I}$.

Finally, since $M(I)\cap M(J)$ is empty for all $I,J\in\sP_{n,m}$ with $I\neq J$, then $\{m_I\mid I\in\sP_{n,m}\}$ is linearly independent. So, as a consequence of \eqref{011} and \eqref{007}, we get that it is a basis of $\sNCSym_{n,m}$.
\end{proof}

\subsection{Product}

Let $K=(K_1,\ldots,K_h)$ and $L=(L_1,\ldots,L_k)$ be partial set supercompositions. The \emph{over product} of $K$ with $L$ is the partial set supercomposition $K\slash L=K\cup L[n]$, where $n=\deg(K)$. Observe that,\[\deg(K/L)=\deg(K)+\deg(L),\qquad\df(K/L)=\df(K)+\df(L),\qquad\ell(K/L)=\ell(K)+\ell(L).\]For instance, below the over product of two set superpartitions:\[(\{0\},\{0,4\},\{1,2,3\})\slash(\{0\},\{0,3\},\{1,2\})=(\{0\},\{0,4\},\{0\},\{0,7\},\{1,2,3\},\{5,6\}).\]We will denote by $I\shuffle J$ the collection of nontrivial partial set supercompositions obtained from $I/J$, possibly by $\oplus$--operating some pairs of blocks $I_i\oplus J_j[n]$ maintaining the order, of the fermionic blocks, inherited from $I/J$. Note that $K\sqsupseteq I/J$ for all $K\in I\shuffle J$, and $I/J\in I\shuffle J$ if and only if $\{0\}\not\in I\cap J$. For instance, if $I=(\{0\},\{0,3\},\{1,2\})$ and $J=(\{{\blue0},{\blue2}\},\{{\blue1}\})$, the collection $I\shuffle J$ is formed by the following elements.\[\begin{array}{c}
(\{0\},\{0,3\},\{{\blue0},{\blue5}\},\{1,2\},\{{\blue4}\})\\[0.2cm]
\begin{array}{c}
(\{0,{\blue4}\},\{0,3\},\{{\blue0},{\blue5}\},\{1,2\})\\
(\{0\},\{0,3,{\blue4}\},\{{\blue0},{\blue5}\},\{1,2\})\\
(\{0\},\{0,3\},\{{\blue0},1,2,{\blue5}\},\{{\blue4}\})\\
(\{0\},\{0,3\},\{{\blue0},{\blue5}\},\{1,2,{\blue4}\})\\
\end{array}\qquad\begin{array}{c}
(\{0,{\blue4}\},\{0,3\},\{{\blue0},1,2,{\blue5}\})\\
(\{0\},\{0,3,{\blue4}\},\{{\blue0},1,2,{\blue5}\})\\
(\{0,{\blue4}\},\{0,3\},\{{\blue0},1,2,{\blue5}\})\\
(\{0\},\{0,3,{\blue4}\},\{{\blue0},1,2,{\blue5}\})
\end{array}
\end{array}\]

As it was shown in \cite[Proposition 3.3]{BeReRoZa08} for $\NCSym$. In superspace, the collection of monomial functions, indexed by set superpartitions, can be described as follows.

\begin{pro}\label{023}
For every pair of set superpartitions $I,J$, we have\[m_Im_J=\sum_{K\in I\shuffle J}m_K.\]
\end{pro}
\begin{proof}
Let $(n,m)$ and $(p,q)$ be the bidegrees of $I$ and $J$ respectively. As $x$--variables commute with $\theta$--variables, we have\[\begin{array}{rcl}
m_Im_J&=&{\displaystyle\left(\sum_{u\in M(I)}\underbrace{\theta_{a_1}\cdots\theta_{a_m}x_{b_1}\cdots x_{b_n}}_u\right)\!\left(\sum_{v\in M(J)}\underbrace{\theta_{c_1}\cdots\theta_{c_q}x_{d_1}\cdots x_{d_p}}_v\right)},\\[0.8cm]
&=&{\displaystyle\sum_{uv\in M(I)M(J)}\underbrace{\theta_{a_1}\cdots\theta_{a_m}\theta_{c_1}\cdots\theta_{c_q}x_{b_1}\cdots x_{b_n}x_{d_1}\cdots x_{d_p}}_{uv}}.\end{array}\]Observe that, in the sum above, the coefficient of $uv\in M(I)M(J)$ is zero, if and only if, either $a_i=c_j$ for some $i,j$, or there are $i,j$ such that $a_i\neq b_k$ for all $k\in[n]$ and $c_j\neq d_k$ for all $k\in[p]$.

Assume that the coefficient of $uv$ is nonzero, and let $K$ be the partial set supercomposition such that $uv\in M(K)$. If $\{a_1,\ldots,a_m,b_1,\ldots,b_n\}\cap\{c_1,\ldots,c_q,d_1,\ldots,d_p\}$ is empty, then $K=I/J$. Otherwise, we have $a_i=d_j$ or $b_i=c_j$ or $b_i=d_j$ for some $i,j$. Then, the block of $K$ containing $i$ and $j+n$ is given by $A\cup B[n]$, where $A$ is the block of $I$ containing $i$, and $B$ is the block of $J$ containing $j$. As the coefficient of $uv$ is nonzero, then $A,B$ are not fermionic at once. Therefore $K\in I\shuffle J$.

Conversely, by definition each $K\in I\shuffle J$ maintains the order of the fermionic blocks of $I/J$. This implies that every $w\in M(K)$ can be written as a product $w=uv$, where $u\in M(I)$ and $v\in M(J)$.
\end{proof}

\section{Other bases}\label{047}

In this section, we introduce power sum, elementary and complete homogeneous symmetric functions in superspace, which constitute bases of $\NCSym$. Also, we give explicit relations between these bases.

\subsection{Power sum functions in superspace}

Here we introduce the basis of power sum functions in superspace.

\begin{dfn}\label{008}
The \emph{power sum function} of a partial set supercomposition $K=(K_1,\ldots,K_q)$ of bidegree $(n,m)$ is the function in superspace, defined as follows:\[p_K\,\,\,\,=\sum_{(a_1,\ldots,a_m,b_1,\ldots,b_n)}\theta_{a_1}\cdots\theta_{a_m}x_{b_1}\cdots x_{b_n},\]where\begin{enumerate}
\item For each $i\in[m]$, $a_i=b_j$ if $j\in K_i^+$.\label{009}
\item For each $i,j\in[n]$, $b_i=b_j$ if $i,j$ belong to the same block of $K$.\label{013}
\end{enumerate}
We will denote by $P(K)$ the set of all monomials satisfying conditions \eqref{009} and \eqref{013} above.
\end{dfn}
For instance, for the partial set supercomposition $K=(\{0,2\},\{0,1,3\},\{4\})$ in $\sC_{4,2}$, we have
\[p_K=\theta_2\theta_1x_1x_2x_1^2+\theta_1\theta_2x_2x_1x_2x_1+\theta_2\theta_1x_1x_2x_1x_2+\theta_1\theta_2x_2x_1x_2^2+\theta_3\theta_2x_2x_3x_2^2+\theta_2\theta_3x_3x_2x_3x_2+\cdots\]

\begin{pro}\label{030}
For every nontrivial partial set supercomposition $K$ of bidegree $(n,m)$, we have\[p_K=\sum_{L\,\sqsupseteq\,K}m_L.\]
\end{pro}
\begin{proof}
Let $K\in\sC_{n,m}$. We must show that $P(K)$ is the union of $M(L)$ such that $L\sqsupseteq K$.

Let $u=\theta_{a_1}\cdots\theta_{a_m}x_{b_1}\cdots x_{b_n}\in P(K)$. Since $K$ has at most one trivial fermionic block, then $u$ is non null--symmetric. So, by Proposition \ref{006}\eqref{012}, there is a unique nontrivial $L\in\sC_{n,m}$ such that $u\in M(L)$. By Definition \ref{008}\eqref{009} and Definition \ref{005}\eqref{010}, for each $i\in[m]$, we have $a_i=b_j$ for all $j\in K_i^+$. Thus $K_i\subseteq L_i$. On the other hand, Definition \ref{008}\eqref{011} and Definition \ref{005}\eqref{011} imply that each nonfermionic block of $K$ is contained in one of $L$. Therefore $L\sqsupseteq K$.

Conversely, let $L\in\sC_{n,m}$ such that $L\sqsupseteq K$, and let $u=\theta_{a_1}\cdots\theta_{a_m}x_{b_1}\cdots x_{b_n}\in M(L)$. By definition, for each $i\in[m]$, we have $K_i\subseteq L_i$ and $a_i=b_j$ for all $j\in L_i^+$. Thus $a_i=b_j$ for all $j\in K_i^+$. Further, if $b_i,b_j$ belong to the same block of $K$ for some $i,j\in[n]$, then $b_i,b_j$ belong to the same block of $L$ because $L\sqsupseteq K$. These facts imply that $u$ satisfies conditions \eqref{009} and \eqref{013} of Definition \ref{008}. Therefore $u\in P(K)$.
\end{proof}

\begin{rem}\label{041}
As $I\preceq J$ for some $I,J\in\sP_{n,m}$ implies that each fermionic block of $J$ contains a unique fermionic block of $I$, we may consider $\perm{I,\,J}$ to be the permutation given by these contentions, that is, $I_i\subseteq J_{\perm{I,\,J}(i)}$ for all $i\in[m]$. The number of inversions of $\perm{I,\,J}$ will be denoted by $\inv(I,J)$. For instance,\[I=(\stackrel{1}{\{0\}},\stackrel{2}{\{0,3\}},\stackrel{3}{\{0,5\}},\{1\},\{2,4\}),\quad J=(\stackrel{1}{\{{\red0},1,{\red5}\}},\stackrel{2}{\{{\red0},2,4\}},\stackrel{3}{\{{\red0},{\red3}\}}),\quad\perm{I,\,J}=231,\quad\inv(I,J)=2.\]Observe that $I\preceq J$ for some $I,J\in\sP_{n,m}$ do not imply that $I\sqsubseteq J$, however $\sigma\fcirc I\sqsubseteq J$ and $I\sqsubseteq \sigma^{-1}\fcirc J$, where $\sigma=\sigma_{_{(I,J)}}$. Moreover if $I\sqsubseteq K\sqsubseteq\sigma^{-1}\fcirc J$ for some $K\in\sC_{n,m}$, then $\sigma\fcirc I\sqsubseteq\sigma\fcirc K\sqsubseteq J$. Similarly if $\sigma\fcirc I\sqsubseteq K\subseteq J$, then $I\sqsubseteq\sigma^{-1}\fcirc K\sqsubseteq\sigma^{-1}\fcirc J$. For instance, $I=(\{0\},\{0,1\},\{2\})$ is not strongly finer than $J=(\{0,1\},\{0,2\})$, see Figure \ref{025}, however, $21\fcirc I\sqsubseteq J$ and $I\sqsubseteq21\fcirc J$.
\end{rem}

Due to Proposition \ref{030} and Remark \ref{041} we obtain the following consequence.
\begin{crl}\label{014}
For every set superpartition $I$ of bidegree $(n,m)$, we have\[p_I=\sum_{J\,\succeq I}(-1)^{\inv(I,\,J)}m_J.\]
\end{crl}

\subsection{Elementary functions in superspace}

Here we introduce the basis of elementary functions in superspace.

\begin{dfn}\label{016}
The \emph{elementary function} of a partial set supercomposition $K=(K_1,\ldots,K_q)$ of bidegree $(n,m)$ is the function in superspace, defined as follows:\[e_K\,\,\,\,=\sum_{(a_1,\ldots,a_m,b_1,\ldots,b_n)}\theta_{a_1}\cdots\theta_{a_m}x_{b_1}\cdots x_{b_n},\]where\begin{enumerate}
\item For each $i\in[m]$, $a_i\neq b_j$ if $j\in K_i^+$.\label{017}
\item For each $i,j\in[n]$, $b_i\neq b_j$ if $b_i,b_j$ belong to the same block of $K$.\label{018}
\end{enumerate}
We will denote by $E(K)$ the set of all monomials satisfying conditions \eqref{017} and \eqref{018} above.
\end{dfn}

For instance, for the set superpartition $I=(\{0\},\{0,2\},\{1,3\})$ in $\sP_{3,2}$, we have
\begin{gather}\label{073}e_I=\theta_1\theta_2x_1x_1x_2+\theta_2\theta_1x_2x_2x_1+\theta_2\theta_1x_1x_2x_2+\theta_1\theta_2x_2x_1x_1+\theta_2\theta_3x_1x_1x_2+\theta_3\theta_2x_1x_1x_2+\cdots\end{gather}Observe that due to the anticommuting of the fermionic variables, some monomials will cancel out in the sum.

\begin{pro}\label{043}
For every $K=(K_1,\ldots,K_q)\in\sC_{n,m}$, we have\[e_K\,\,\,=\!\sum_{K\,\sqcap\,L\,=\,\hat{0}_{n,m}}\!\!m_L.\]
\end{pro}
\begin{proof}
Let $L=(L_1,\ldots,L_r)\in\sC_{n,m}$ satisfying $K\sqcap L=\hat{0}_{n,m}$. Proposition \ref{022} implies that $K_i\cap L_i=\{0\}$ if $i\in[m]$ and $K^+\wedge L^+=\hat{0}$. Let $u=\theta_{a_1}\cdots\theta_{a_m}x_{b_1}\cdots x_{b_n}$ be a monomial in $M(L)$. If $a_i=b_j$ for some $i\in[m]$ and $j\in K_i^+$, Definition \ref{005}\eqref{010} implies that $j\in L_i$, which is impossible because $K_i\cap L_i=\{0\}$. Thus $a_i\neq b_j$. Now, if $b_i=b_j$ for some $i,j$ belonging to a block of $K$, Definition \ref{005}\eqref{002} implies that $i,j$ belong to the same block of $L$, which contradicts the fact that $K^+\wedge L^+=\hat{0}$. Thus $b_i\neq b_j$. Therefore $u\in E(K)$.

Conversely, consider $u=\theta_{a_1}\cdots\theta_{a_m}x_{b_1}\cdots x_{b_n}$ be a monomial in $E(K)$. Since $K$ has at most one trivial fermionic block, then $u$ is non null--symmetric. So, by Proposition \ref{006}\eqref{012}, there is a unique $L\in\sC_{n,m}$ such that $u\in M(L)$. Let $j\in K_i\cap L_i$ for some $i\in[m]$. If $j>0$, Definition \ref{005}\eqref{010} implies that $a_i=b_j$ which contradicts Definition \ref{016}\eqref{017}. Thus $j=0$. Now, if $i,j$ belong to the same block of $K^+\wedge L^+$, Definition \ref{005}\eqref{002} implies that $b_i=b_j$ which contradicts Definition \ref{016}\eqref{018}. Therefore $K^+\wedge L^+=\hat{0}$.
\end{proof}
As a consequence, by using Proposition \ref{006}\eqref{007}, we have the following characterization in terms of set superpartitions.
\begin{crl}
For every set superpartition $I=(I_1,\ldots,I_k)$ with $\df(I)=m$, we have\[e_I=\sum_{(\sigma,\,J)}(-1)^{\inv(\sigma)}m_J,\]where $J=(J_1,\ldots,J_q)$ with $\df(J)=m$ satisfying $I^+\wedge J^+=\hat{0}$ and $I_i\cap J_{\sigma(i)}=\{0\}$ for all $i\in[m]$.
\end{crl}

For instance, for the set superpartition in \eqref{073} we have
\[\begin{array}{rcl}e_I&=&m_{_{(\{0,1,2\},\{0,3\})}}-m_{_{(\{0\},\{0,1,2\},\{3\})}}-m_{_{(\{0,1\},\{0,2,3\})}}-m_{_{(\{0\},\{0,2,3\},\{1\})}}\\&&\textcolor{white}{.}-m_{_{(\{0\},\{0,2\},\{1\},\{3\})}}-m_{_{(\{0,1\},\{0,2\},\{3\})}}+m_{_{(\{0,2\},\{0,3\},\{1\})}}.\end{array}\]

\subsection{Complete homogeneous functions in superspace}\label{034}

Here we introduce the basis of complete homogeneous functions in superspace.

\begin{dfn}\label{019}
The \emph{complete homogeneous function} of a partial set supercomposition $K=(K_1,\ldots,K_q)$ of bidegree $(n,m)$ is the function in superspace, defined as follows:\[h_K\,\,\,\,=\sum_uc_uu\]where every $u$ is a non null-symmetric monomial with $m$ $\theta$--variables and $n$ $x$--variables, and $c_u$ is the number of ways to sort the elements of the blocks of $K\sqcap L$ simultaneously, with $L$ the unique partial set supercomposition of bidegree $(n,m)$ satisfying $u\in M(L)$.

We will denote by $H(I)$ the set of all monomials satisfying the conditions of Definition \ref{019}.
\end{dfn}

For instance, for the set superpartition $I=(\{0\},\{0,2\},\{1,3\})$ in $\sP_{3,2}$, we have
\begin{gather}\label{074}
h_I=2\theta_1\theta_2x_2x_2x_2+2\theta_2\theta_1x_1x_1x_1+\theta_2\theta_1x_1x_1x_2+\theta_1\theta_2x_2x_2x_1+2\theta_1\theta_2x_1x_2x_1+2\theta_2\theta_1x_2x_1x_2+\cdots
\end{gather}

For every $A=(A_1,\ldots,A_k)$, with $A_1,\ldots,A_k$ subsets of $[n]_0$, we define $A!=|A_1|!\cdots|A_k|!$.
\begin{pro}\label{065}
For every partial set supercomposition $K$, we have\[h_K=\sum_L(K\sqcap L)!m_L,\]where $L$ is a partial set supercomposition with the same bidegree of $K$.
\end{pro}
\begin{proof}
It is a direct consequence of fact that $c_u=(K\sqcap L)!$ and $c_u=c_v$ for all $u,v\in M(L)$.
\end{proof}

For $I,J\in\sP_{n,m}$ and $\sigma\in\Sym_m$, we set $I\wedge_{\sigma}J=I\sqcap(\sigma\fcirc J)$. By Proposition \ref{022}, we can describe $I\wedge_{\sigma}J$ as the tuple $(N_1,\ldots,N_q)$, where $N_i=I_i\cap J_{\sigma(i)}$ for all $i\in[m]$ and $(N_{m+1},\ldots,N_q)=(I\backslash A)\wedge(J\backslash A)$ with $A=N_1\cup\cdots\cup N_m$. For instance, if $I=(\{0,1,3\},\{0,2\},\{4,5\})$ and $J=(\{0,2,4\},\{0,3\},\{1\},\{5\})$, then:\[A=\{0,2,3\},\quad I\backslash A=(\{1\},\{4,5\}),\quad J\backslash A=(\{1\},\{4\},\{5\}),\quad I\wedge_{21}J=(\{0,3\},\{0,2\},\{1\},\{4\},\{5\}).\]Thus, we obtain the following consequence.
\begin{crl}\label{031}
For every set superpartition $I\in\sP_{n,m}$, we have\[h_I=\sum_{(\sigma,\,J)\in\Sym_m\times\sP_{n,m}}(I\wedge_{\sigma}J)!(-1)^{\inv(\sigma)}m_J.\]
\end{crl}

For instance, for the set superpartition in \eqref{074}, we have\[\begin{array}{rcl}
h_I&=&2m_{_{(\{0\},\{0,1,2,3\})}}-m_{_{(\{0,1,2\},\{0,3\})}}+2m_{_{(\{0,1,3\},\{0,2\})}}+m_{_{(\{0,1\},\{0,2,3\})}}+m_{_{(\{0\},\{0,1,2\},\{3\})}}\\&&\textcolor{white}{.}+m_{_{(\{0\},\{0,2,3\},\{1\})}}+2m_{_{(\{0\},\{0,2\},\{1,3\})}}+m_{_{(\{0\},\{0,2\},\{1\},\{3\})}}+m_{_{(\{0,1\},\{0,2\},\{3\})}}-m_{_{(\{0,2\},\{0,3\},\{1\})}}.
\end{array}\]

\subsection{Relationship between bases}

Here we show that the sets of power sum functions, elementary functions and homogeneous functions are indeed bases of the algebra $\sNCSym$.

As it occurs with monomial functions, for each $(\sigma,I)\in\Sym_m\times\sP_{n,m}$, we have\begin{equation}\label{020}p_{\sigma\fcirc I}=(-1)^{\inv(\sigma)}p_I,\qquad e_{\sigma\fcirc I}=(-1)^{\inv(\sigma)}e_I,\qquad h_{\sigma\fcirc I}=(-1)^{\inv(\sigma)}h_I.\end{equation}

We begin by writing monomial functions in terms of sums of power sum functions, elementary functions and homogeneous functions. For this, we will use the lattice structure of partial set supercompositions.

The \emph{M\"{o}bius function} \cite[Section 3.7]{St97} of $(\sC_{n,m},\sqsubseteq)$ is the map $\mu:\sC_{n,m}\times\sC_{n,m}\to\field$ such that $\mu(K,K)=1$, and\[\mu(K,L)=-\sum_{K\,\sqsubseteq\,N\,\sqsubset\,L}\mu(K,N)\quad\text{if}\quad K\sqsubset L.\]Observe that\begin{equation}\label{044}\sum_{K\,\sqsubseteq\,N\,\sqsubseteq\,L}\mu(K,N)=\delta_{K,L}\end{equation}A well known formula in terms of \emph{chains} can be given for the M\"{o}bius function $\mu$. Indeed, given two partial set supercompositions $K,L$, a \emph{chain} joining $K$ with $L$ is a finite sequence of partial set supercompositions $F_1\sqsubset\cdots\sqsubset F_q$ such that $F_1=K$ and $F_q=L$. A chain as above will be denoted by $[F_1,\ldots,F_q]$. Thus,\begin{equation}\label{029}\mu(K,L)=\sum_{[F_1,\ldots,\,F_q]}(-1)^{q-1},\end{equation}where the sum is taken over all chains joining $K$ with $L$. For instance, if $K=\hat{0}_{2,2}$ and $L=(\{0\},\{0,1,2\})$ as in Figure \ref{025}, there are three chains joining $K$ with $L$, that is,\[[K,L],\qquad[K,(\{0\},\{0,2\},\{1\}),L],\qquad[K,(\{0\},\{0\},\{1,2\}),L],\qquad[K,(\{0\},\{0,1\},\{2\}),L].\]Thus $\mu(K,L)=(-1)+3(-1)^2=2$.

Proposition \ref{027} and \cite[Section 7, Proposition 3]{Ro64} imply that, for every $K=(K_1,\ldots,K_l)\in\sC_{n,m}$, we have the following formulas:\begin{equation}\label{028}\mu(\hat{0}_{n,m},K)=\prod_{i=1}^l(-1)^{|K_i|-1}(|K_i|-1)!\quad\text{and}\quad\sum_{K\sqsubseteq L}|\mu(\hat{0}_{n,m},K)|=L!\end{equation}The \emph{sign} $(-1)^K$ of $K$ is defined as $\prod_{i=1}^l(-1)^{|K_i|-1}$. Observe that $(-1)^{K/L}=(-1)^K(-1)^L$, and\begin{equation}\label{045}\mu(\hat{0}_{n,m},K)=(-1)^K|\mu(\hat{0}_{n,m},K)|.\end{equation}

The following proposition is a consequence of Equation \eqref{028} and the inversion M\"{o}bius formula \cite[Section 3, Proposition 2, Proposition 3]{Ro64}. The proof is analogous to the one of \cite[Theorem 3.3]{RoSa06}.
\begin{pro}\label{015}
For every set superpartition $I$ of bidegree $(n,m)$, we have
\[m_I=\sum_{K\,\sqsupseteq\,I}\mu(I,K)p_K,\qquad m_I=\sum_{L\,\sqsupseteq\,I}\frac{\mu(I,L)}{\mu(\hat{0}_{n,m},L)}\sum_{K\,\sqsubseteq\,L}\mu(K,L)e_K,\qquad m_I=\sum_{L\,\sqsupseteq\,I}\frac{\mu(I,L)}{|\mu(\hat{0}_{n,m},L)|}\sum_{K\,\sqsubseteq\,L}\mu(K,L)h_K,\]where $K,L$ are partial set supercompositions of bidegree $(n,m)$.
\end{pro}

\begin{thm}
The sets of power sum functions, elementary functions and complete homogeneous functions indexed by set superpartitions are bases of the algebra $\sNCSym$.
\end{thm}
\begin{proof}
It is a consequence of Proposition \ref{015}, Equation \ref{020} and the fact the functions are null whenever they are indexed by trivial partial set supercompositions.
\end{proof}
Observe that, specifically, the formula for power sum functions in Proposition \ref{015} can be written as follows\[m_I=\sum_{I\preceq J}\mu(I,J)(-1)^{\inv(I,J)}p_J,\qquad I,J\in\sP_{n,m}.\]

\begin{exm}
By Figure \ref{025}, Equation \ref{020}, Equation \ref{029} and Equation \ref{028}, if $I=(\{0\},\{0,1\},\{2\})$, Proposition \ref{015} implies the following\[\begin{array}{c}
m_I=p_{_{(\{0\},\{0,1\},\{2\})}}-p_{_{(\{0\},\{0,1,2\})}}+p_{_{(\{0,1\},\{0,2\})}},\\[1.5mm]
m_I=\frac{1}{2}e_{_{(\{0\},\{0,1\},\{2\})}}-\frac{1}{2}e_{_{(\{0\},\{0,2\},\{1\})}}-\frac{1}{2}e_{_{(\{0\},\{0,1,2\})}}+e_{_{(\{0,1\},\{0,2\})}},\\[2mm]
m_I=\frac{5}{2}h_{_{(\{0\},\{0,1\},\{2\})}}-\frac{1}{2}h_{_{(\{0\},\{0,2\},\{1\})}}-\frac{1}{2}h_{_{(\{0\},\{0,1,2\})}}+h_{_{(\{0,1\},\{0,2\})}}.
\end{array}\]
\end{exm}

The following result is a consequence of Proposition \ref{015} and the inversion M\"{o}bius formula.
\begin{pro}\label{035}
For every set superpartition $I$ of bidegree $(n,m)$, we have:\[\begin{array}{rclrcl}
e_I&=&{\displaystyle\sum_{K\,\sqsubseteq\,I}\mu(0_{n,m},K)p_K},&\kern+4em p_I&=&{\displaystyle\frac{1}{\mu(0_{n,m},I)}\sum_{K\,\sqsubseteq\,I}\mu(K,I)e_K},\\[0.7cm]
h_I&=&{\displaystyle\sum_{K\,\sqsubseteq\,I}|\mu(0_{n,m},K)|p_K},&\kern+4em p_I&=&{\displaystyle\frac{1}{|\mu(0_{n,m},I)|}\sum_{K\,\sqsubseteq\,I}\mu(K,I)h_K},\\[0.7cm]
e_I&=&{\displaystyle\sum_{K\,\sqsubseteq\,I}(-1)^K\sum_{L\,\sqsubseteq\,K}\mu(L,K)h_L},&\kern+4em h_I&=&{\displaystyle\sum_{K\,\sqsubseteq\,I}(-1)^K\sum_{L\,\sqsubseteq\,K}\mu(L,K)e_L}.
\end{array}\]where $K$ is a partial supercomposition.
\end{pro}

As in the classic case, by using Proposition \ref{035}, we may define a linear map $\omega:\sNCSym\to\sNCSym$ satisfying $\omega(e_I)=h_I$ for all $I\in\sP_n$. Due to \eqref{020}, for every $K=\sigma\fcirc I$ with $(\sigma,I)\in\Sym_m\times\sP_{n,m}$, we have\begin{equation}\label{046}\omega(e_K)=\omega(e_{\sigma\fcirc I})=\omega((-1)^{\inv(\sigma)}e_I)=(-1)^{\inv(\sigma)}\omega(e_I)=(-1)^{\inv(\sigma)}h_I=h_{\sigma\fcirc I}=h_K.\end{equation}By using Proposition \ref{035} and \eqref{046}, we have $\omega^2=1$ and $\omega(p_I)=(-1)^Ip_I$.

As it was shown in \cite[Lemma 4.1]{BeHoRoZa06} and \cite[Proposition 2.39 and Corollary 2.41]{Be-B10} for $\NCSym$. In superspace, the bases of power sum, elementary and complete homogeneous functions are multiplicative.

\begin{pro}\label{042}
For set superpartitions $I\in\sP_{n,m}$ and $J\in\sP_{p,q}$, we have
\[p_Ip_J=p_{I/J},\qquad e_Ie_J=e_{I/J},\qquad h_Ih_J=h_{I/J}.\]
\end{pro}
\begin{proof}
First, we show the multiplicativity of power sum functions. Let $A=\{H\mid H\in K\shuffle L,\,K\succeq I,\,L\succeq J\}$. By definition, $H\succeq I/J$ for all $H\in A$. Conversely, let $H=(H_1,\ldots,H_k)\succeq I/J$. Note that two different fermionic blocks of $I/J$ cannot be contained in the same fermionic block of $H$. Then, for each $i\in[k]$, we define the following sets:\begin{gather*}K_i=\left\{\begin{array}{ll}
H_i\cap[n]_0&\text{if $H_i$ contains a fermionic block of $I$,}\\
H_i\cap[n]&\text{otherwise}.\end{array}\right.\\[0.2cm]
L_i=\left\{\begin{array}{ll}
H_i\cap[n+1,n+p]_0&\text{if $H_i$ contains a fermionic block of $J[n]$,}\\
H_i\cap[n+1,n+p]&\text{otherwise}.
\end{array}\right.\end{gather*}
Observe that $H_i=K_i\sqcup L_i$ for all $i\in[k]$. This implies that $H\in K\shuffle L$, where $K$ and $L$ are the set superpartitions obtained by removing the empty components from $(K_1,\ldots,K_k)$ and $(L_1,\ldots,L_k)$ respectively. Further $K\succeq I$ and $L\succeq J$. Thus, by applying Proposition \ref{023}, we obtain the following\[p_Ip_J=\sum_{K\succeq I}\sum_{L\succeq J}m_Km_L=\sum_{K\succeq I}\sum_{L\succeq J}\left(\sum_{H\in K\shuffle L}m_H\right)=\sum_{H\in A}m_H=\sum_{H\succeq I/J}m_H=p_{I/J}.\]

To show the multiplicativity of elementary functions, we first observe that if $K\sqsubseteq I$ and $L\sqsubseteq J$, then $K/L\sqsubseteq I/J$. Conversely, if $H\sqsubseteq I/J$, we consider $K\in\sC_n$ obtained by removing the blocks of $H$ contained in blocks of $J[n]$, and $L\in\sC_p$ obtained by $(-n)$--shifting the collection produced by removing the blocks of $H$ contained in blocks of $I$. Then $H=K/L$. Thus, by applying Proposition \ref{035}, and since the M\"{o}bius function is multiplicative, we obtain the following\[e_Ie_J=\sum_{K\,\sqsubseteq\,I}\sum_{L\,\sqsubseteq\,J}\mu(0_{n,m},K)\mu(0_{p,q},L)p_{K/L}=\sum_{H\,\sqsubseteq\,I/J}\mu(0_{n,m}/0_{p,q},K/L)p_H=\sum_{H\,\sqsubseteq\,I/J}\mu(0_{n+p,m+q},H)p_H=e_{I/J}.\]Finally, to show the multiplicativity of complete homogeneous functions, we first note that $\omega$ is an algebra morphism. By regarding power sun functions, we have\[\omega(p_Kp_L)=\omega(p_{K/L})=(-1)^{K/L}p_{K/L}=(-1)^Kp_K(-1)^Lp_L=\omega(p_K)\omega(p_L).\]Therefore $h_Ih_J=\omega(e_I)\omega(e_J)=\omega(e_Ie_J)=\omega(e_{I/J})=h_{I/J}$.
\end{proof}

\section{Commuting variables and inner product}

In this section, we show formulas that project our bases in their commuting counterpart, and we prove that there is an isometry from $\sNCSym$ to $\sSym$. 

Let $\field^{\theta}[x]=\sseries/I$, where $I$ is the ideal of $\sseries$ generated by the elements $x_ix_j-x_jx_i$ for all $i,j\in\N$. We denote by $\rho$ the associated natural projection $\sseries\to\kern-0.8em\to\field^{\theta}[x]$. The image of $\sNCSym$ under this projection is the algebra $\sSym$ of symmetric functions in commuting variables in superspace, introduced in \cite{DeLaMa01}. The bases of this algebra are indexed by the so--called superpartitions.

\subsection{Symmetric functions in commuting variables in superspace}\label{080}

Recall that bases of the classic symmetric functions in commuting variables are indexed by partitions of numbers \cite{Mac99,St99}. A \emph{partition} $\lambda$ of \emph{degree} $n=:\deg(\lambda)$ is a decreasing sequence $(\lambda_1,\ldots,\lambda_k)$ of nonnegative integers, called \emph{components}, such that $\lambda_1+\cdots+\lambda_k=n$. In general, every finite sequence of nonnegative integers $u$ can be converted into a partition $\sort(u)$ by sorting its components in decreasing order. The partition $\lambda$ can also be denoted by $(1^{r_1},\ldots,n^{r_n})$, where $r_i$ is the \emph{multiplicity} of $i$, that is, the number of times that $i$ occurs in $\lambda$. We set $\lambda!=\lambda_1!\cdots\lambda_k!$ and $\lambda^!=r_1!\cdots r_n!$. For instance, if $\lambda=(5,5,5,3,3,1,1,1,1)$, then\[\lambda=(1^4,3^3,5^2),\qquad\lambda!=5!\,5!\,5!\,3!\,3!,\qquad\lambda^!=4!\,3!\,2!.\]

As we mentioned above, the bases of $\sSym$ are indexed by superpartitions \cite{DeLaMa06}. A \emph{superpartition} $\Lambda$ of degree $n$ \cite[Subsection 4.2.2]{DeLaMa01} is an array $(\Lambda_1,\ldots,\Lambda_m\,;\Lambda_{m+1},\ldots,\Lambda_k)$ of nonnegative integers, called \emph{components}, satisfying $\Lambda_1+\cdots+\Lambda_k=n$, where $\Lambda^a:=(\Lambda_1,\ldots,\Lambda_m)$ is a partition of different components and $\Lambda^s:=(\Lambda_{m+1},\ldots,\Lambda_k)$ is a usual partition. Components of $\Lambda^a$ are called \emph{fermionic components} of $\Lambda$, and the numbers $m,k$ are called the \emph{fermionic degree} and the \emph{length} of $\Lambda$, respectively. Set $\Lambda!=\Lambda^a!\Lambda^s!$. The degree, fermionic degree and length of $\Lambda$ will be denoted as with set superpartitions, that is, $\deg(\Lambda)$, $\df(\Lambda)$ and $\ell(\Lambda)$ respectively. The collection of superpartitions of degree $n$ will be denoted by $\sPar_n$, and the subcollection of elements with fermionic degree $m$ will be denoted by $\sPar_{n,m}$. As with usual partitions, every array of nonnegative integers $u=(u_1,\ldots,u_m;u_{m+1},\ldots,u_k)$, with $u_1,\ldots,u_m$ all different, can be converted into a superpartition $\sort(u)$, where $\sort(u)^a=\sort(u_1,\ldots,u_m)$ and $\sort(u)^s=\sort(u_{m+1},\ldots,u_k)$.

Some of the classical symmetric functions in superspace in commutating variables were defined in \cite{DeLaMa06}. Indeed, given a superpartition $\Lambda=(\Lambda_1,\ldots,\Lambda_m;\Lambda_{m+1},\ldots,\Lambda_n)$, we have the following\begin{itemize}
\item\emph{Monomial symmetric functions in superspace}:\[m_{\Lambda}(x,\theta)=\sum_{\sigma\in\Sym_n} \theta_{\sigma(1)}\cdots\theta_{\sigma(m)}x_{\theta(1)}^{\Lambda_1}\cdots x_{\sigma(n)}^{\Lambda_n},\]where the sum considers no repeated terms.
\item\emph{Power--sum symmetric functions in superspace}: $p_{\Lambda}=\tp_{\Lambda_1}\cdots\tp_{\Lambda_m}p_{\Lambda_{m+1}}\cdots p_{\Lambda_n}$, where\[\tp_k=\sum_{i=1}^n\theta_ix_i^k\quad\text{and}\quad p_r=\sum_{i=1}^nx_i^r,\quad\text{for}\quad k\geq0, r\geq1.\]
\item\emph{Elementary symmetric functions in superspace}: $e_{\Lambda}=\te_{\Lambda_1}\cdots\te_{\Lambda_m}e_{\Lambda_{m+1}}\cdots e_{\Lambda_n}$, where\[\te_k=m_{(0;1^k)}\quad\text{and}\quad e_r=m_{(\emptyset;1^r)},\quad\text{for}\quad k\geq0, r\geq1.\] 
\item\emph{Complete homogeneous symmetric functions in superspace}: $h_{\Lambda}=\th_{\Lambda_1}\ldots\th_{\Lambda_m}h_{\Lambda_{m+1}}\cdots h_{\Lambda_n}$, where\[\th_k=\sum_{|\Lambda|=k,\,\df(\Lambda)=1}(\Lambda_1+1)m_{\Lambda}\quad\text{ and}\quad h_r=\sum_{|\Lambda|=r,\,\df(\Lambda)=0}m_\Lambda,\quad\text{for}\quad k\geq0, r\geq1.\]
\end{itemize}
Observe that the bases of $\SSym$ are recovered whenever $\Lambda^a=\emptyset$.

For every set superpartition $I=(I_1,\ldots,I_k)$ of bidegree $(n,m)$, which has no fermionic blocks of the same size, we define the super partition:\[\Lambda(I)=\sort(|I_1|-1,\ldots,|I_m|-1;|I_{m+1}|,\ldots,|I_k|).\]Given a superpartition $\Lambda$, we say that $I$ is \emph{of type} $\Lambda$ if $\Lambda(I)=\Lambda$. We denote by $\varepsilon(I)$ the minimal number of transpositions needed to obtain $\Lambda(I)^a$ from $(|I_1|-1,\ldots,|I_m|-1)$. For instance, the set superpartitions $(\{0,1,7\},\{0,5\},\{2\},\{3,4\},\{6\})$ and $(\{0,2\},\{0,3,5\},\{1\},\{4\},\{6,7\})$ are of type $(2,1;2,1,1)$.

To determine the number of set superpartitions of type $\Lambda=(\Lambda_1,\ldots,\Lambda_m;\Lambda_{m+1},\ldots,\Lambda_k)$, we set $\Lambda^{\!+}=\sort(\Lambda_1,\ldots,\Lambda_k)$. Note that $\Lambda!=\Lambda^{\!+}!$.

\begin{lem}\label{037}
Let $\Lambda=(\Lambda_1,\ldots,\Lambda_m;\Lambda_{m+1},\ldots,\Lambda_k)$ be a superpartition of degree $n$, and for each $i\in[k]$ denote by $k_i$ the multiplicity of $\Lambda_i$ in $\Lambda^{\!+}$. Then $(\Lambda^{\!+})^!=k_1\cdots k_m(\Lambda^s)^!$.
\end{lem}
\begin{proof}
The result is a direct consequence of the fact that, for $i>m$, the multiplicity of $\Lambda_i$ in $\Lambda^s$ is either $k_i-1$ whenever $\Lambda_i\in\Lambda^a$, or it is $k_i$ otherwise.
\end{proof}

Recall that, for a partition $\lambda$ of degree $n$, the number of set partitions of type $\lambda$ is $\binom{n}{\lambda}=\frac{n!}{\lambda!\lambda^!}$. Below we generalize this result for superpartitions.

\begin{lem}\label{038}
Let $\Lambda=(\Lambda_1,\ldots,\Lambda_m;\Lambda_{m+1},\ldots,\Lambda_k)$ be a superpartition of degree $n$, and for each $i\in[m]$ denote by $k_i$ the multiplicity of $\Lambda_i$ in $\Lambda^{\!+}$. Then, the number of set superpartitions of type $\Lambda$ is\[\binom{n}{\Lambda}:=k_1\cdots k_m\binom{n}{\,\,\Lambda^{\!+}}=\frac{n!}{\Lambda!(\Lambda^s)^!}.\]
\end{lem}
\begin{proof}
Observe that every set superpartition of type $\Lambda$ can be obtained from a usual set partition of type $\Lambda^{\!+}$ by appending zeros to some of its blocks. Indeed, if $I$ is a set superpartition of type $\Lambda$, then $J=I^+$ is a set partition of type $\Lambda^{\!+}$. Conversely, if $J$ is a set partition of type $\Lambda^+$, then there are $B_1,\ldots,B_m$ in $J\cup\{\emptyset\}$ such that $|B_i|=\Lambda_i$ for all $i\in[m]$. We get a set superpartition $I=K\backslash\{\emptyset\}$ of type $\Lambda$, where $K$ is obtained from $J\cup\{\emptyset\}$ by replacing $B_i$ by $B_i\cup\{0\}$. The proof follows from the fact that for each $\Lambda_i$ with $i\in[m]$, there are $k_i$ blocks $B_i\in J\cup\{\emptyset\}$ such that $|B_i|=\Lambda_i$ together with Lemma \ref{037}.
\end{proof}

Now, we will study the image of the bases of $\sNCSym$ under the projection $\rho$ onto the algebra of commuting variables $\sSym$. First, we prove that set superpartitions needed to describe $\rho$ on the basis of monomial functions are the ones whose fermionic blocks have all different sizes.

\begin{pro}
Let $I$ be a set superpartition having at least two fermionic blocks $B,B'$ such that $|B|=|B'|$. Then $\rho(m_I)=0$.
\end{pro}
\begin{proof}
Let $k=\ell(I)$, and let $i,j\in[k]$ such that $B,B'$ are the $i$th and $j$th blocks of $I$ respectively. Consider $g\in\{m_I,p_I\}$, and let $r=\theta_{a_1}^{\epsilon_1}\cdots\theta_{a_k}^{\epsilon_k}x_{b_1}\cdots x_{b_n}$ be a monomial occurring in $g$. Then $r=\epsilon\theta_{a_i}\theta_{a_j}\theta_{d_1}\cdots\theta_{d_r}x_{b_1}\cdots x_{b_n}$ for some $\epsilon\in\{\pm1\}$, where $r=\df(I)-2$. Since $g$ is symmetric, the monomial $\sigma(r)=-\epsilon\theta_{a_i}\theta_{a_j}\theta_{d_1}\cdots\theta_{d_r}x_{\sigma(b_1)}\cdots x_{\sigma(b_n)}$ also occurs in $g$, where $\sigma$ is the transposition exchanging $a_i$ with $a_j$. Since $|B|=|B'|$, then $x_{\sigma(b_1)}\cdots x_{\sigma(b_n)}=x_{b_1}\cdots x_{b_n}$ in $\field^{\theta}[x]$. Thus $\rho(\sigma(r))=-\rho(r)$. As it holds for every monomial occurring in $g$, then $\rho(g)=0$.
\end{proof}

\begin{lem}\label{021}
For every $K\in\sC_{n,m}$ and $\delta\in\Sym_m$, we have $\rho(b_{\delta K})=\rho(b_K)$, where $b$ denotes any of the basis elements studied above.
\end{lem}
\begin{proof}
Since $\delta\circ b_K=b_{\delta\circ K}$, then, every monomial occurring in $b_{\delta\circ K}$ can be written as\[\delta\circ(\theta_{a_1}\cdots\theta_{a_m}x_{b_1}\cdots x_{b_n})=\theta_{a_1}\cdots\theta_{a_m}x_{b_{\delta^{-1}(1)}}\cdots x_{b_{\delta^{-1}(n)}},\]where $\theta_{a_1}\cdots\theta_{a_m}x_{b_1}\cdots x_{b_n}$ is a monomial occurring in $b_K$. Now, since $x$--variables become commuting under the action of $\rho$, then $\rho(\theta_{a_1}\cdots\theta_{a_m}x_{b_{\delta^{-1}(1)}}\cdots x_{b_{\delta^{-1}(n)}})=\rho(\theta_{a_1}\cdots\theta_{a_m}x_{b_1}\cdots x_{b_n})$. Thus, $\rho(b_{\delta\circ K})=\rho(b_K)$.
\end{proof}

\begin{pro}\label{039}
For every set superpartition $I$ with fermionic blocks of different sizes, we have
\begin{enumerate}
\item $\rho(m_I)=(-1)^{\varepsilon(I)}(\Lambda(I)^s)^!m_{\Lambda(I)}$.\label{040}
\item $\rho(p_I)=(-1)^{\varepsilon(I)}p_{\Lambda(I)}$.
\item $\rho(e_I)=(-1)^{\varepsilon(I)}\Lambda(I)!e_{\Lambda(I)}$.
\item $\rho(h_I)=(-1)^{\varepsilon(I)}\Lambda(I)!h_{\Lambda(I)}$.
\end{enumerate}
Furthermore, if $J$ is a set superpartition with blocks of the same size, we have $\rho(b_J)=0$, where $b$ is any of the basis elements above.
\end{pro}
\begin{proof}
Due to Lemma \ref{021}, it is enough to take $I=(I_1,\ldots,I_k)\in\sP_{n,m}$ as a convex set superpartition. 

To prove (1), consider $\Lambda(I)^s=(1^{k_1},\ldots,n^{k_n})$. Since $I$ is convex, every monomial $u\in M(I)$ can be written as\[u=\theta_{a_1}\cdots\theta_{a_m}x_{a_1}^{|I_1|-1}\cdots x_{a_m}^{|I_m|-1}x_{b_{m+1}}^{|I_{m+1}|}\cdots x_{b_k}^{|I_k|}.\]Observe that for every permutation $\sigma$ satisfying $\sigma(\{b_{m+1},\ldots,b_k\})=\{b_{m+1},\ldots,b_k\}$ with $|I_i|=|I_j|$ whenever $\sigma(b_i)=b_j$, and $\sigma(b)=b$ for all $b\not\in\{b_{m+1},\ldots,b_k\}$, we have\[\rho(\theta_{a_1}\cdots\theta_{a_m}x_{a_1}^{|I_1|-1}\cdots x_{a_m}^{|I_m|-1}x_{\sigma(b_{m+1})}^{|I_{m+1}|}\cdots x_{\sigma(b_k)}^{|I_k|})=\rho(u).\]Since, for each $i\in[n]$, there are $k_i$ blocks of size $i$ in $(I_{m+1},\ldots,I_k)$, so there are $k_1!\cdots k_n!=(\Lambda(I)^s)^!$ such permutations. Finally, by sorting the $\theta$--variables, we obtain $\rho(m_I)=(-1)^{\varepsilon(I)}(\Lambda(I)^s)^!m_{\Lambda(I)}$.

To prove (2) to (4) we will use the fact that the bases are multiplicative (Proposition \ref{042}).

Observe that $\rho(p_{[n]})=p_n$ and $\rho(p_{[n]_0})=\tp_n$, where $p_{[n]}=p_{\{[n]\}}$ and $p_{[n]_0}=p_{\{[n]_0\}}$. For $i\in[k]$, let $n_i=|I_i|-1$ if $i\leq m$, and $n_i=|I_i|$ otherwise. Since $I$ is convex and the basis is multiplicative, we have\[\rho(p_I)=\rho(p_{[n_1]_0}\cdots p_{[n_m]_0}p_{[n_{m+1}]}\cdots p_{[n_k]})=\tp_{n_1}\cdots\tp_{n_m}p_{n_{m+1}}\cdots p_{n_k}=(-1)^{\varepsilon(I)}p_{\Lambda(I)}.\]The argument is similar for the remaining bases by observing that\[\rho(e_{[n]})=n!e_n,\qquad\rho(e_{[n]_0})=n!\te_n,\qquad\rho(h_{[n]})=n!h_n,\qquad\rho(h_{[n]_0})=n!\th_n.\]
\end{proof}

We may define a \emph{lifting map} $\tilde{\rho}:\sSym\to\sNCSym$, which result to be a right inverse of $\rho$, as the linear map satisfying
\begin{equation}\label{061}\tilde{\rho}(m_{\Lambda})=\frac{\Lambda!}{n!}\sum_{\Lambda(I)=\Lambda}(-1)^{\varepsilon(I)}m_I.\end{equation}

\begin{pro}\label{062}
The composition $\rho\circ\tilde{\rho}$ is the identity map on $\sSym$.
\end{pro}
\begin{proof}
By \eqref{040} of Proposition \ref{039} together with Lemma \ref{037} and Lemma \ref{038}, we have\[\rho(\tilde{\rho}(m_{\Lambda}))=\frac{\Lambda!}{n!}\sum_{\Lambda(I)=\Lambda}(-1)^{\varepsilon(I)}\rho(m_I)=\frac{\Lambda!}{n!}\sum_{\Lambda(I)=\Lambda}(\Lambda(I)^s)^!m_{\Lambda(I)}=\frac{\Lambda!}{n!}\binom{n}{\Lambda}(\Lambda(I)^s)^!m_{\Lambda}=m_{\Lambda}.\]Therefore $\rho\circ\tilde{\rho}$ is the identity map on $\sSym$.
\end{proof}

\subsection{Inner product}\label{078}

It was shown in \cite{DeLaMa06} that $\sSym$ can be endowed with an inner product $\ipro{\,,}$ satisfying $\langle m_{\Lambda},h_{\Omega}\rangle=(-1)^{\binom{m}{2}}\delta_{\Lambda,\Omega}$, where $\Lambda,\Omega$ are superpartitions, $m=\df(\Lambda)$ and $\delta$ is the Dirac delta function.

Here we extend the inner product of $\NCSym$ \cite{RoSa06} to superspace and show that, with respect to this inner product, $\tilde{\rho}$ is an isometry from $\sNCSym$ to $\sSym$.

Let $\ipro{\,,}$ be the bilinear form on $\sNCSym$ satisfying $\ipro{m_I,h_J}=(-1)^{\binom{m}{2}}n!\delta_{I,J}$, where $I,J$ are set superpartitions and $(n,m)$ is the bidegree of $I$.

\begin{lem}\label{033}
For every $I,J\in\sP_{n,m}$ and $\sigma\in\Sym_m$, we have $(I\wedge_{\sigma}J)!=(J\wedge_{\sigma^{-1}}I)!$.
\end{lem}
\begin{proof}
Let $(K_1,\ldots,K_p)=I\wedge_{\sigma}J$, and let $(K_1',\ldots,K_q')=J\wedge_{\sigma^{-1}}I$. By definition, we have $K'_{\sigma(i)}=J_{\sigma(i)}\cap I_{\sigma^{-1}(\sigma(i))}=J_{\sigma(i)}\cap I_i=I_i\cap J_{\sigma(i)}=K_i$. This implies that $K_1\cup\cdots\cup K_m=K_1'\cup\cdots\cup K_m'=:A$. Hence, $(K_{m+1},\ldots,K_p)=(K_{m+1}',\ldots,K_q')=(I\backslash A)\wedge(J\backslash A)$ with $p=q$. Therefore $(I\wedge_{\sigma}J)!=(J\wedge_{\sigma^{-1}}I)!$.
\end{proof}

\begin{pro}\label{032}
Let $I,J$ be two set superpartitions of bidegree $(n,m)$. Then\[\ipro{h_I,h_J}=n!\sum_{\sigma\in\Sym_m}(I\wedge_{\sigma}J)!(-1)^{\inv(\sigma)}\]Furthermore, $\ipro{h_I,h_J}=\ipro{h_J,h_I}$.
\end{pro}
\begin{proof}
Due to Corollary \ref{031}, we have\[\ipro{h_I,h_J}=\ipro{\sum_{(\sigma,\,K)}(I\wedge_{\sigma}K)!(-1)^{\inv(\sigma)}m_K,h_J}=\sum_{(\sigma,K)}(I\wedge_{\sigma}K)!(-1)^{\inv(\sigma)}\ipro{m_K,h_J},\]where $\sigma\in\Sym_m$ and $J\in\sP_n$. Then, because of $\ipro{m_K,h_J}=(-1)^{\binom{m}{2}}n!\delta_{K,\,J}$, we obtain the result\[\ipro{h_I,h_J}=\sum_{(\sigma,K)}(I\wedge_{\sigma}K)!(-1)^{\inv(\sigma)}(-1)^{\binom{m}{2}}n!\delta_{K,\,J}=(-1)^{\binom{m}{2}}n!\sum_{\sigma\in\Sym_m}(I\wedge_{\sigma}K)!(-1)^{\inv(\sigma)}.\]Lemma \ref{033} and the fact that $\inv(\sigma)$ and $\inv(\sigma^{-1})$ have the same parity imply that $\ipro{h_I,h_J}=\ipro{h_J,h_I}$.
\end{proof}

\begin{pro}\label{036}
Let $I,J$ be two set superpartitions of bidegree $(n,m)$. Then\[\ipro{\,p_I,p_J}=\frac{(-1)^{\binom{m}{2}}n!}{|\mu(\hat{0}_{n,m},I)|}\delta_{I,J}.\]
\end{pro}
\begin{proof}
Due to Corollary \ref{014} and Proposition \ref{035}, we have
\[p_I=\sum_{K\succeq I}m_K,\qquad p_J=\frac{1}{|\mu(\hat{0}_{n,m},J)|}\sum_{L\preceq J}\mu(L,J)h_L,\]
where $K,L$ are nontrivial partial set supercompositions of bidegree $(n,m)$. Then,\[\ipro{\,p_I,p_J}=\frac{1}{|\mu(\hat{0}_{n,m},J)|}\sum_{K\succeq I}\sum_{L\preceq J}\mu(L,J)\ipro{m_K,h_L}.\]Since $K,L$ are nontrivial, then $K=(\sigma,\bar{K})$ and $L=(\tau,\bar{L})$ for some $\sigma,\tau\in\Sym_m$. Thus\[\ipro{m_K,h_L}=(-1)^{\inv(\sigma)+\inv(\tau)}\ipro{m_{\bar{K}},h_{\bar{L}}}.\]We will distinguish four cases. If $I\succ J$ as set superpartitions, then $\bar{L}\preceq J\prec I\preceq\bar{K}$. Thus $\ipro{m_K,h_L}=0$. Therefore $\ipro{\,p_I,p_J}=0$.

If $I,J$ are not comparable as set superpartitions, either $\bar{K}$ is not comparable with $J$, or $\bar{K}\succ J$, otherwise $I\preceq\bar{K}\preceq J$, which is a contradiction. Thus, $\mu(K,J)h_K$ is not an addend of $p_J$, which implies $\ipro{\,p_I,p_J}=0$.

Now, if $I\prec J$ as set superpartitions, Remark \ref{041} implies that $I\prec(\alpha^{-1},J)$, as partial set supercomposition, for some $\alpha\in\Sym_m$.\[\begin{array}{rcl}
\ipro{\,p_I,p_J}&=&{\displaystyle\frac{1}{|\mu(\hat{0}_{n,m},J)|}\sum_{(\alpha,\,I)\preceq L\preceq J}\mu(L,J)\ipro{m_{(\alpha^{-1},L)},h_L}},\\[0.5cm]
&=&{\displaystyle\frac{1}{|\mu(\hat{0}_{n,m},J)|}\sum_{(\alpha,\,I)\preceq L\preceq J}(-1)^{\inv(\alpha)}\mu(L,J)\ipro{m_L,h_L}},\\[0.5cm]
&=&{\displaystyle\frac{(-1)^{\inv(\alpha)}(-1)^{\binom{m}{2}}n!}{|\mu(\hat{0}_{n,m},J)|}\sum_{(\alpha,I)\preceq L\preceq J}\mu(L,J)=0.}
\end{array}\]Therefore ${\displaystyle\ipro{\,p_I,p_J}=\frac{(-1)^{\binom{m}{2}}n!}{|\mu(\hat{0}_{n,m},I)|}}\delta_{I,J}$.
\end{proof}

\begin{lem}\label{064}
For every $K,L\in\sC_{n,m}$ and $(\sigma,\tau)\in\Sym_n\times\Sym_m$, we have $(K\sqcap L)!=\left(\tau\fcirc\sigma(K)\,\sqcap\,\tau\fcirc\sigma(L)\right)!$.
\end{lem}
\begin{proof}
Write $K=(K_1,\ldots,K_h)$ and $L=(L_1,\ldots,L_k)$. Let $A=(K_1\cap L_1)\cup\cdots\cup(K_m\cap L_m)$, and let $(K^+\backslash A)\wedge(L^+\backslash A))=(N_{m+1},\ldots,N_q)$, then $(\sigma(K)^+\backslash\sigma(A))\wedge(\sigma(L)^+\backslash\sigma(A)))=(\sigma(N_{m+1}),\ldots,\sigma(N_q))$. Since $\sigma,\tau$ are bijections, then $|K_i\cap L_i|=|\sigma(K_{\tau(i)})\cap\sigma(L_{\tau(i)})|$ for all $i\in[m]$, and $|\sigma(N_i)|=|N_i|$ for all $i\in[m+1,q]$. So, Proposition \ref{022} implies that $(K\sqcap L)!=\left(\tau\fcirc\sigma(K)\,\sqcap\,\tau\fcirc\sigma(L)\right)!$.
\end{proof}

\begin{pro}\label{063}
The bilinear form $\ipro{\,,}$ is an inner product on $\sNCSym$. Furthermore, with this inner product, the lifting map $\tilde{\rho}$ is an isometry from $\sNCSym$ to $\sSym$.
\end{pro}
\begin{proof}
Proposition \ref{032} and Proposition \ref{036} imply that $\ipro{\,,}$ is an inner product. To show that $\tilde{\rho}$ is an isometry, for every $\Lambda\in\sPar_{n,m}$, we define\[H_{\Lambda}=\sum_{\Lambda(I)=\Lambda}(-1)^{\varepsilon(I)}h_I,\quad\text{where}\quad I\in\sP_{n,m}.\]By Proposition \ref{039}(4) and Lemma \ref{038}, we have\begin{equation}\label{066}\rho(H_{\Lambda})=\sum_{\Lambda(I)=\Lambda}(-1)^{\varepsilon(I)}\rho(h_I)=\sum_{\Lambda(I)=\Lambda}\Lambda(I)!h_{\Lambda(I)}=\sum_{\Lambda(I)=\Lambda}\Lambda!h_{\Lambda}=\binom{n}{\Lambda}\Lambda!h_{\Lambda}=\frac{n!}{(\Lambda^s)^!}h_{\Lambda}.\end{equation}On the other hand, Proposition \ref{065} implies that\[H_{\Lambda}=\sum_{\Lambda(I)=\Lambda}(-1)^{\varepsilon(I)}\sum_{K\in\sC_{n,m}}(I\sqcap K)!m_K=\sum_{K\in\sC_{n,m}}\underbrace{\sum_{\Lambda(I)=\Lambda}(-1)^{\varepsilon(I)}(I\sqcap K)!}_{C_K}m_K=\sum_{K\in\sC_{n,m}}C_Km_K.\]By Lemma \ref{064}, we have $C_K=C_L$ for all $K,L\in\sC_{n,m}$ with $\bar{K}$ and $\bar{L}$ of the same type. Hence,\[H_{\Lambda}=\sum_{\Omega\in\sPar_{n,m}}C_\Omega\tilde{\rho}(m_\Omega)=\tilde{\rho}\left(\underbrace{\sum_{\Omega\in\sPar_{n,m}}C_\Omega m_\Omega}_F\right)=\tilde{\rho}(F),\quad\text{where}\quad C_\Omega=\sum_{\Lambda(\bar{K})=\Omega}C_K.\]By using \eqref{066} together with Proposition \ref{062}, we get ${\displaystyle F=\rho(\tilde{\rho}(F))=\rho(H_\Lambda)=\frac{n!}{(\Lambda^s)^!}h_{\Lambda}}$. Thus\[\begin{array}{rcl}
\ipro{\tilde{\rho}(m_{\Omega}),\tilde{\rho}(h_{\Lambda})}&=&{\displaystyle\ipro{\frac{\Omega!}{n!}\sum_{\Lambda(I)=\Omega}(-1)^{\varepsilon(I)}m_I,\frac{(\Lambda^s)^!}{n!}H_{\Lambda}}},\\[0.5cm]
&=&{\displaystyle\ipro{\frac{\Omega!}{n!}\sum_{\Lambda(I)=\Omega}(-1)^{\varepsilon(I)}m_I,\frac{(\Lambda^s)^!}{n!}\sum_{\Lambda(J)=\Lambda}(-1)^{\varepsilon(J)}h_J}},\\[0.5cm]
&=&{\displaystyle\frac{\Lambda!(\Lambda^s)^!}{(n!)^2}\binom{n}{\Lambda}(-1)^{\binom{m}{2}}n!}\delta_{\Omega,\Lambda},\\[0.4cm]
&=&\ipro{m_\Omega,h_\Lambda}.
\end{array}\]Therefore $\tilde{\rho}$ is an isometry.
\end{proof}

As shown in the proof of Proposition \ref{063}, we have ${\displaystyle\tilde{\rho}(h_\Lambda)=\frac{(\Lambda^s)^!}{n!}H_\Lambda}$. Similarly, we can show that\begin{equation}\label{076}\tilde{\rho}(p_\Lambda)=\binom{n}{\Lambda}^{-1}P_\Lambda,\quad\text{where}\quad P_\Lambda=\sum_{\Lambda(I)=\Lambda}(-1)^{\varepsilon(I)}p_I.\end{equation}Indeed, due to Proposition \ref{030}, we get\[P_\Lambda=\sum_{\Lambda(I)=\Lambda}(-1)^{\varepsilon(I)}\sum_{K\sqsupseteq I}m_K.\]Also, observe that if $m_K$ and $m_L$ are monomial symmetric functions occurring in $P_\Lambda$ such that $K,L$ are partial set supercompositions whose $i$th fermionic blocks contain the same number of elements with $K^+,L^+$ of the same type, then their coefficient in $P_\Lambda$ are the same as well. This implies that $P_\Lambda$ belongs to the image of $\tilde{\rho}$. Now, \eqref{076} follows by applying Proposition \ref{062}.

\section{Schur symmetric functions in superspace}

In this section, we define Schur--type functions in noncommuting variables in superspace by means super semistandard Young tableaux. Also, we express these functions in terms of the monomial basis and show their projection on Schur--type functions in commuting variables in superspace.

Recall that partitions of integers can be represented by \emph{Young tableaux}. Indeed, the \emph{Young diagram} or simply the \emph{diagram} of a partition $\lambda=(\lambda_1,\ldots,\lambda_k)$ is the array with $k$ rows of left--justified boxes in which the $i$th row contains exactly $\lambda_i$ boxes. See Figure \ref{048}.
\begin{figure}[H]
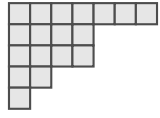

\figfou
\caption{Diagram of the partition $\lambda=(7,4,4,2,1)$.}
\label{048}
\end{figure}

A partition $\mu$ is said to be \emph{contained} in a partition $\lambda$, denoted by $\mu\subseteq\lambda$, if $\ell(\mu)\leq\ell(\lambda)$ and $\mu_i\leq\lambda_i$ for all $i\in[\ell(\mu)]$. If $\mu\subseteq\lambda$, the \emph{skew diagram} $\lambda/\mu$ is the diagram obtained by removing the boxes from the diagram of $\lambda$ that belong to the diagram of $\mu$. The skew diagram $\lambda/\mu$ is called an \emph{horizontal $r$--strip} (resp. \emph{vertical $r$--strip}) if $\deg(\lambda)-\deg(\mu)=r$ and each column (resp. row) has at most one box. See Figure \ref{049}.
\begin{figure}[H]
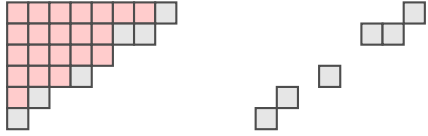

\figfiv
\caption{A horizontal $7$--strip $\lambda/\mu$, where $\lambda=(8,7,5,4,2,1)$ and $\mu=(7,5,5,3,1)$.}
\label{049}
\end{figure}
In this section, we identify $\Lambda=(\Lambda_1,\ldots,\Lambda_m;\Lambda_{m+1},\ldots,\Lambda_k)\in\sPar_{n,m}$ with the pair $(\Lambda^\oplus,\Lambda^+)$, where\[\Lambda^\oplus=\sort(\Lambda_1+1,\ldots,\Lambda_m+1,\Lambda_{m+1},\ldots,\Lambda_k).\]Observe that the skew diagram $\Lambda^\oplus/\Lambda^+$ is simultaneously an horizontal $m$--strip and an vertical $m$--strip. The superpartition $\Lambda$ can be represented either by the pair of diagrams of $(\Lambda^\oplus,\Lambda^+)$, or by the diagram of $\Lambda^\oplus$ in which the boxes of $\Lambda^\oplus/\Lambda^+$ has been replaced by circles. These diagrams, accepting at most one round box at the end of each row, will be called \emph{super Young diagram}. See Figure \ref{050}. The \emph{conjugate} of a superpartition $\Lambda$ is the superpartition $\Lambda'$ obtained by conjugating both $\Lambda^\oplus$ and $\Lambda^+$.
\begin{figure}[H]
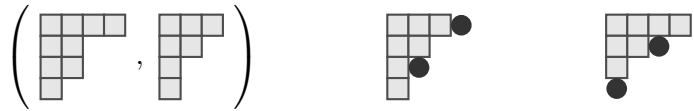

\figsix
\caption{Diagrams of $\Lambda=(3,1;2,1)$ with $\Lambda^+=(3,2,1,1)$ and $\Lambda^\oplus=(4,2,2,1)$, and its conjugate.}
\label{050}
\end{figure}
Consider $\bar{\N}_0=\{\bar{0},\bar{1},\bar{2},\ldots\}$. For every $\alpha\in\N_0\cup\bar{\N}_0$, we set $|\alpha|=a\in\N_0$, if either $\alpha=a$ or $\alpha=\bar{a}$.

Let $\Lambda$ be a superpartition, and let $\alpha_1,\ldots,\alpha_r\in\N_0\cup\bar{\N}_0$. A \emph{super semistandard Young tableau} of \emph{shape} $\Lambda$ and \emph{weight} $(\alpha_1,\ldots,\alpha_r)$, is a sequence of superpartitions $(\Lambda_{(0)},\ldots,\Lambda_{(r)})$ with $\Lambda_{(0)}=\emptyset$ and $\Lambda_{(r)}=\Lambda$, satisfying:
\begin{enumerate}
\item For each $i\in[n]$, $\Lambda_{(i)}^+/\Lambda_{(i-1)}^+$ is a horizontal $|\alpha_i|$--strip.
\item If $\alpha_i\in\N_0$ and $\Lambda_{(i-1)}$ contains the $j$th circle in the $r$th row, then\label{051}
\begin{enumerate}
\item $\df(\Lambda_{(i)})=\df(\Lambda_{(i-1)})$.
\item $\Lambda_{(i)}$ contains the $j$th circle in the $r$th row whenever the $r$th row of $\Lambda_{(i)}^+/\Lambda_{(i-1)}^+$ is empty.
\item $\Lambda_{(i)}$ contains the $j$th circle in the $(r+1)th$ row whenever the $r$th row of $\Lambda_{(i)}^+/\Lambda_{(i-1)}^+$ is nonempty.
\end{enumerate}
\item If $\alpha_i\in\bar{\N}_0$, then 
\begin{enumerate}
\item $\df(\Lambda_{(i)})=\df(\Lambda_{(i-1)})+1$.
\item $\Lambda_{(i)}$ contains a unique circle, say in the $c$th column, such that the $c$th column of $\Lambda^+_{(i)}/\Lambda^+_{(i-1)}$ is empty and every other column to its left is nonemtpy.
\item The diagram obtained from $\Lambda_{(i)}$ by removing the new circle above satisfies conditions in \eqref{051}.
\end{enumerate}
\end{enumerate}
See Figure \ref{052} for an instance.
\begin{figure}[H]
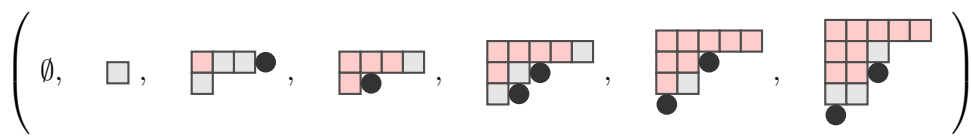

\figsev
\caption{A super semistandard Young tableau of shape $(1,\bar{3},1,\bar{3},1,3)$.}
\label{052}
\end{figure}
Every super semistandard Young tableau $(\Lambda_{(0)},\ldots,\Lambda_{(r)})$ of weight $(\alpha_1,\ldots,\alpha_r)$ and shape $\Lambda$ can be represented by filling with positive integers $a_1<\cdots<a_r$ the boxes and circles of $\Lambda$, such that for each $i\in[r]$, the boxes of $\Lambda^+_{(i)}/\Lambda^+_{(i-1)}$ in $\Lambda$ are filled with the integer $a_i$, and if $\alpha_i\in\bar{\N}_0$, the new circle of $\Lambda_{(i)}$ in $\Lambda$ is filled with $a_i$ as \mbox{well. See Figure \ref{053}.} Observe that if $\Lambda$ is a usual partition, this representation coincides with the classic definition of semistandard Young tableaux \cite{St99}.
\begin{figure}[H]
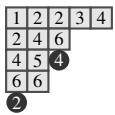

\figeig
\caption{Super semistandard Young tableau in Figure \ref{052} represented by filling its boxes with numbers.}
\label{053}
\end{figure}

If $T$ is a super Young diagram, we will denote its boxes by $T_1^s,\ldots,T_n^s$, where $i<j$ if and only if $T_j^s$ is located either to the right or below $T_i^s$. The circles of $T$ are denoted by $T_1^a,\ldots,T_m^a$, indexed from top to bottom. Given $t$ either a box or a circle of $T$, we denote by $c(t)$ its label.

\begin{dfn}
The \emph{Schur function} of a superpartition $\Lambda\in\sPar_{n,m}$ is the function in superspace, defined as follows:
\[S_\Lambda=\sum_{T\in sSSYT(\Lambda)}\underbrace{\theta_{_{c(T^a_1)}}\cdots\theta_{_{c(T^a_m)}}\sum_{\sigma\in\Sym_n}x_{_{c(T^s_{\sigma(1)})}}\cdots x_{_{c(T^s_{\sigma(n)})}}}_{[\theta x]^T},\]where $sSSYT(\Lambda)$ is the collection of super semistandard Young tableux of shape $\Lambda$.
\end{dfn}
For instance, for $\Lambda=(2,1;\,\,)$ and $T=\fignin$ in $sSSYT(\Lambda)$ with weight $(\bar{1},\bar{2})$, we have\[[\theta x]^T=\theta_2\theta_1(x_1x_2x_2+x_2x_1x_2+x_2x_2x_1+x_1x_2x_2+x_2x_1x_2+x_2x_2x_1).\]
Observe that if $\Lambda\in\sPar_{n,m}$ and $T\in sSSYT(\Lambda)$ with weight $(\alpha_1,\ldots,\alpha_r)$, then there is $\sigma\in\Sym_n$ such that\[x_{_{c(T^s_{\sigma(1)})}}\cdots x_{_{c(T^s_{\sigma(n)})}}=x_1^{|\alpha_1|}\cdots x_r^{|\alpha_r|}=:x^T.\]Thus,\begin{equation}\label{056}[\theta x]^T=\theta^T\sum_{\sigma\in\Sym_n}\sigma^{-1}\circ x^T,\quad\text{where}\quad\theta^T=\theta_{_{c(T^a_1)}}\cdots\theta_{_{c(T^a_m)}}.\end{equation}

\begin{lem}\label{057}
For every superpartition $\Lambda$, there is a bijection between the subcollection of $sSSYT(\Lambda)$ with weight $(\alpha_1,\ldots,\alpha_r)$ and the subcollection of $sSSYT(\Lambda)$ with weight $(\alpha_1,\ldots,\alpha_{i-1},\alpha_{i+1},\alpha_i,\alpha_{i+1},\ldots,\alpha_r)$.
\end{lem}
\begin{proof}
Let $g$ be the classic bijection between semistandard Young tableaux of shape $\lambda$ and weight $(a_1,\ldots,a_r)$ and semistandard Young tableaux of shape $\lambda$ and weight $(a_1,\ldots,a_{i-1},a_{i+1},a_i,a_{i+1},\ldots,a_r)$ \cite[Theorem 7.10.2]{St99}.

Now, let $G$ be the map sending a super semistandard Young tableau $(\Lambda_{(0)},\ldots,\Lambda_{(r)})$ of shape $\Lambda$ and weight $(\alpha_1,\ldots,\alpha_r)$ to the super semistandard Young tableau $(\Lambda_{(0)},\ldots,\Lambda_{(i-1)},\Omega,\Lambda_{(i+1)},\ldots,\Lambda_{(r)})$ of shape $(\alpha_1,\ldots,\alpha_{i-1},\alpha_{i+1},\alpha_i,\alpha_{i+1},\ldots,\alpha_r)$, where $g(\Lambda_{(0)}^+,\ldots,\Lambda_{(r)}^+)=(\Lambda_{(0)}^+,\ldots,\Lambda_{(i-1)}^+,\Omega^+,\Lambda_{(i+1)}^+,\ldots,\Lambda_{(r)}^+)$, such that the number of new circles from $\Lambda_{(i-1)}$ to $\Omega$ is the same as the one from $\Lambda_{(i)}$ to $\Lambda_{(i+1)}$, and similarly the number of new circles from $\Omega$ to $\Lambda_{(i+1)}$ is the same as the one from $\Lambda_{(i-1)}$ to $\Lambda_{(i)}$. As $g$ is a bijection and new circles are uniquely determined, then $G$ is a bijection as well.
\end{proof}

\begin{lem}\label{058}
For every monomial $u$ in $x$--variables, we have $s_i(\sigma^{-1}\circ u)=\sigma^{-1}\circ(s_i(u))$, where $s_i=(i\,\,i+1)$, $\sigma\in\Sym_n$ and $\circ$ is the action defined in \eqref{055}.
\end{lem}
\begin{proof}
Let $u=x_{b_1}\cdots x_{b_n}$. It is enough to take $\sigma=(j\,\,j+1)$. Then:\[\begin{array}{rcl}
s_i(\sigma^{-1}\circ u)&=&s_i(x_{b_1}\cdots x_{b_{j-1}}x_{b_{j+1}}x_{b_j}x_{b_{j+2}}\cdots x_{b_n}),\\
&=&x_{s_i(b_1)}\cdots x_{s_i(b_{j-1})}x_{s_i(b_{j+1})}x_{s_i(b_j)}x_{s_i(b_{j+2})}\cdots x_{s_i(b_n)},\\
&=&\sigma^{-1}\circ(x_{s_i(b_1)}\cdots x_{s_i(b_n)}),\\
&=&\sigma^{-1}\circ s_i(u).
\end{array}\]
\end{proof}

\begin{thm}\label{069}
For every $\Lambda\in\sPar_{n,m}$, the Schur function $S_\Lambda$ is symmetric.
\end{thm}
\begin{proof}
We need to show that for every $T\in sSSYT(\Lambda)$ there is $T'\in sSSYT(\Lambda)$ such that $s_i([\theta x]^T)=[\theta x]^{T'}$ with $s_i=(i\,\,i+1)$ for all $i$. Lemma \ref{057} implies that there are $\tau\in\Sym_n$ and $T'\in sSSYT(\Lambda)$, such that $s_i(\theta^T)=\theta^{T'}$ and $s_i(x^T)=\tau^{-1}\circ x^{T'}$. Thus, by applying Lemma \ref{058}, we have \[s_i([\theta x]^T)=s_i(\theta^T)\sum_{\sigma\in\Sym_n}s_i(\sigma^{-1}\circ x^T)=\theta^{T'}\sum_{\sigma\in\Sym_n}\sigma^{-1}\circ s_i(x^T)=\theta^{T'}\sum_{\sigma\in\Sym_n}\sigma^{-1}\circ(\tau^{-1}\circ x^{T'})
=[\theta x]^{T'}.\]
\end{proof}

Given a super semistandard Young $T$ with $m$ circles, we denote by $\inv(T)$, the number of inversions of the word $c(T_1^a)\cdots c(T_m^a)$. As defined in \cite{JoLa17}, for $\Lambda,\Omega\in\sPar_{n,m}$ with $\Omega=(\Omega_1,\ldots,\Omega_m;\Omega_{m+1},\ldots,\Omega_k)$, we set\begin{equation}\label{070}K_{\Lambda,\Omega}=\sum_T(-1)^{\inv(T)},\end{equation}where $T$ is a super semistandard Young tableau of shape $\Lambda$ and weight $(\bar{\Omega}_1,\ldots,\bar{\Omega}_m,\Omega_{m+1},\ldots,\Omega_k)$.

\begin{pro}\label{059}
For every superpartition $\Lambda\in\sPar_{n,m}$, we have\[S_{\Lambda}=\sum_\Omega K_{\Lambda,\Omega}\Omega!\sum_{\Lambda(I)=\Omega}(-1)^{\varepsilon(I)}m_I.\]
\end{pro}
\begin{proof}
By Proposition \ref{006}\eqref{068} and Theorem \ref{069}, there are scalars $c_I$ such that\[S_\Lambda=\sum_{I\in\sP_{n,m}}c_Im_I.\]Let $u\in M(I)$ for some $I\in\sP_{n,m}$. Since $u$ occurs in $[\theta x]^T$ for some $T\in sSSYT(\Lambda)$, then\[u=\theta_{_{c(T^a_1)}}\cdots\theta_{_{c(T^a_m)}}x_{_{c(T^s_{\sigma(1)})}}\cdots x_{_{c(T^s_{\sigma(n)})}},\qquad\sigma\in\Sym_n.\]Observe that $c_I$ is the coefficient of $u$ in $S_\Lambda$, that is, $c_I=d_I\Lambda(I)!$, where $d_I$ is the number of super semistandard Young tableaux $T$ such that, for each $i$, the $i$th part of $\Lambda(I)$ is the number of $i$'s in the boxes of $T$, and $\Lambda(I)!$ is the number of permutations fixing the set $\{j\mid c(T_j^s)=i\}$. Thus,\[S_\Lambda=\sum_I\Lambda(I)!K_{\Lambda,\Lambda(I)}(-1)^{\varepsilon(I)}m_I=\sum_\Omega K_{\Lambda,\Omega}\Omega!\sum_{\Lambda(I)=\Omega}(-1)^{\varepsilon(I)}m_I.\]
\end{proof}

For instance, $\figten$ is the unique super semistandard Young tableau of shape $(2,1;\,)$ and type $(\bar{2},\bar{1})$, thus\[S_{(2,1;\,)}=2\left(m_{\{0,1,2\},\{0,3\}}+m_{\{0,1,3\},\{0,2\}}-m_{\{0,1\},\{0,2,3\}}\right)=2\left(m_{\{0,1,2\},\{0,3\}}+m_{\{0,1,3\},\{0,2\}}+m_{\{0,2,3\},\{0,1\}}\right)\!.\]Observe that Schur functions cannot be positively expanded in terms of monomial functions indexed by set superpartitions. However, as in the example above, such an expansion is possible when using monomial functions indexed by partial set supercompositions.

A remarkable generalization of classic Schur functions are the so--called \emph{Macdonald polynomials} \cite[Chapter VI]{Mac99}, which are symmetric functions with coefficients given by rational functions in parameters $q$ and $t$. As shown in \cite[Chapter VI]{Mac99}, Schur functions can be obtained as a special case of Macdonald polynomials when setting $q=t$. In superspace, Macdonald polynomials involving parameters $q$ and $t$ were defined in \cite{BlDeLaMa12}, and so \emph{Schur--type functions in commuting variables in superspace} can be gotten as special limits on these parameters. More precisely, we are interested in two Schur--type functions $s_\Lambda$ and $\bar{s}_\Lambda$ obtained by setting $q=t=0$ and $q=t=\infty$, respectively. Combinatorial formulas for $s_\Lambda$ and $\bar{s}_\Lambda$ were given in \cite{JoLa17}. Indeed, for a superpartition $\Lambda$, we have:\begin{equation}\label{077}s_\Lambda=\sum_\Omega K_{\Lambda,\Omega}m_\Omega,\qquad \bar{s}_\Lambda=\sum_\Omega\bar{K}_{\Lambda,\Omega}m_\Omega,\end{equation}where $K_{\Lambda,\Omega}$ is defined in \eqref{070}, and $\bar{K}_{\Lambda,\Omega}$ is an integer number related to another kind of super semistandard Young tableaux which will be defined below.

Let $\Lambda$ be a superpartition, and let $\alpha_1,\ldots,\alpha_r\in\N_0\cup\bar{\N}_0$. A \emph{super semistandard Young tableau of the second kind} of \emph{shape} $\Lambda$ and \emph{weight} $(\alpha_1,\ldots,\alpha_r)$, is a sequence of superpartitions $(\Lambda_{(0)},\ldots,\Lambda_{(r)})$ with $\Lambda_{(0)}=\emptyset$ and $\Lambda_{(r)}=\Lambda$, satisfying:
\begin{enumerate}
\item If $\alpha_i\in\N_0$ and $\Lambda_{(i-1)}$ contains the $j$th circle in the $r$th row, then\label{072}
\begin{enumerate}
\item $\df(\Lambda_{(i)})=\df(\Lambda_{(i-1)})$.
\item $\Lambda_{(i)}^+/\Lambda_{(i-1)}^+$ is a horizontal $\alpha_i$--strip.
\item $\Lambda_{(i)}$ contains the $j$th circle in the $r$th row whenever the $r$th row of $\Lambda_{(i)}^+/\Lambda_{(i-1)}^+$ is empty.
\item $\Lambda_{(i)}$ contains the $j$th circle either in the $r$th row or in the $(r+1)th$ row whenever the $r$th row of $\Lambda_{(i)}^+/\Lambda_{(i-1)}^+$ is nonempty.
\end{enumerate}
\item If $\alpha_i\in\bar{\N}_0$, then 
\begin{enumerate}
\item $\df(\Lambda_{(i)})=\df(\Lambda_{(i-1)})+1$.
\item $\Lambda_{(i)}^\oplus/\Lambda_{(i-1)}^\oplus$ is a horizontal $(|\alpha_i|+1)$--strip.
\item The rightmost box of $\Lambda_{(i)}^\oplus/\Lambda_{(i-1)}^\oplus$, say the new circle, also belongs to $\Lambda_{(i)}^\oplus/\Lambda_{(i)}^+$.
\item The diagram obtained from $\Lambda_{(i)}$ by removing the new circle above satisfies conditions in \eqref{072}.
\end{enumerate}
\end{enumerate}
See Figure \ref{071} for an instance.\begin{figure}[H]
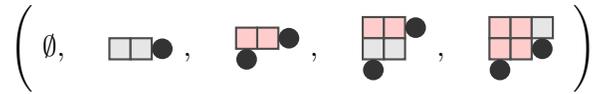
\figele\caption{A super semistandard Young tableau of the second kind of shape $(2,0;3)$ and type $(\bar{2},\bar{0},2,1)$.}\label{071}\end{figure}

As with super semistandard Young tableaux, every super semistandard Young tableau of the second kind can be represented, in the same manner, by filling with numbers the boxes and circles of its Young diagram.

\begin{dfn}
The \emph{Schur function of the second kind} of a superpartition $\Lambda\in\sPar_{n,m}$ is the function in superspace, defined as follows:
\[\bar{S}_\Lambda=\sum_{T\in\bar{s}SSYT(\Lambda)}\theta_{_{c(T^a_1)}}\cdots\theta_{_{c(T^a_m)}}\sum_{\sigma\in\Sym_n}x_{_{c(T^s_{\sigma(1)})}}\cdots x_{_{c(T^s_{\sigma(n)})}},\]where $\bar{s}SSYT(\Lambda)$ is the collection of super semistandard Young tableux of the second kind of shape $\Lambda$.
\end{dfn}

The following proposition is analogous to Theorem \ref{069} and Proposition \ref{059}.
\begin{pro}
For every $\Lambda\in\sPar_{n,m}$, the Schur function of the second kind $\bar{S}_\Lambda$ is symmetric, and\[\bar{S}_{\Lambda}=\sum_\Omega\bar{K}_{\Lambda,\Omega}\Omega!\sum_{\Lambda(I)=\Omega}(-1)^{\varepsilon(I)}m_I.\]
\end{pro}

The following result shows relations, via the projection and the lifting maps, between the Schur functions in noncommuting variables in superspace introduced here and the Schur functions in superspace in \eqref{077}.

\begin{pro}\label{060}
For every superpartitions $\Lambda,\Omega\in\sPar_{n,m}$, we have
\[\begin{array}{rclrclrcl}
\rho(S_{\Lambda})&=&n!s_{\Lambda},\qquad&\tilde{\rho}(n!s_{\Lambda})&=&S_{\Lambda},\qquad&\langle S_\Lambda,S_\Omega\rangle&=&n!^2\ipro{s_\Lambda,s_\Omega},\\[0.3cm]
\rho(\bar{S}_{\Lambda})&=&n!\bar{s}_{\Lambda},\qquad&\tilde{\rho}(n!\bar{s}_{\Lambda})&=&\bar{S}_{\Lambda},\qquad&\langle \bar{S}_\Lambda,\bar{S}_\Omega\rangle&=&n!^2\ipro{s_\Lambda,s_\Omega}.
\end{array}\]
\end{pro}
\begin{proof}
By Proposition \ref{059} and the definition of $\tilde{\rho}$ in \eqref{061}, we have\[S_\Lambda=\sum_{\Omega\preceq\Lambda}K_{\Lambda,\Omega}\Omega!\sum_{\Lambda(I)=\Omega}(-1)^{\sigma(I)}m_I=\sum_{\Omega\preceq\Lambda}K_{\Lambda,\Omega}\Omega!\tilde{\rho}(m_\Omega)\frac{n!}{\Omega!}=n!\sum_{\Omega\preceq\Lambda}K_{\Lambda,\Omega}\,\tilde{\rho}(m_\Omega).\]Now, by applying Proposition \ref{062}, we obtain\[\rho(S_\Lambda)=n!\sum_{\Omega\preceq\Lambda}K_{\Lambda,\Omega}m_\Omega=n!s_\Lambda,\qquad\tilde{\rho}(n!s_\Lambda)=\sum_{\Omega\preceq\Lambda}n!K_{\Lambda,\Omega}\,\tilde{\rho}(m_\Omega)=\sum_{\Omega\preceq\Lambda}K_{\Lambda,\Omega}\Omega! \sum_{\Lambda(I)=\Omega}(-1)^{\varepsilon(I)}m_I=S_\Lambda.\]On the other hand, by applying Proposition \ref{063}, we have\[\ipro{S_\Lambda,S_\Omega}=\ipro{\tilde{\rho}(n!s_\Lambda),\tilde{\rho}(n!s_\Omega)}=n!^2\ipro{s_\Lambda,s_\Omega}.\]The rest of the equalities are proved in the same manner.
\end{proof}

Several classical properties of Schur functions in superspace can be derived from the corresponding properties on the Macdonald polynomials in superspace \cite[Appendix A]{BlDeLaMa12}. Particularly, it was shown that $\{\bar{s}_\Lambda\}_\Lambda$ and $\{s_\Omega\}_\Omega$ are essentially dual to each other, indeed,\[\langle\hat{\omega} (\bar{s}_{\Lambda'}),s_\Omega\rangle=(-1)^{\binom{m}{2}}\delta_{\Lambda\Omega},\]
where the inner product is the one on $\sSym$ defined in Subsection \ref{078}, and $\hat{\omega}:\sSym\to\sSym$ is the involution defined by $\hat{\omega}(p_\Lambda)=(-1)^{\deg(\Lambda)-\ell(\Lambda^s)}p_{\Lambda}$ \cite[Theorem 27]{DeLaMa06}.

Observe that $\{s_\Lambda\}_\Lambda$ is not an orthogonal basis, indeed, we get $\langle s_{(1;0)},s_{(0;1)}\rangle=-\langle h_{(1;\,\,)},m_{(1;\,\,)}\rangle=-1$ because $s_{(0;1)}=-h_{(1;\,\,)}+2h_{(0;1)}$ and $s_{(1;\,\,)}=m_{(1;\,\,)}$. Proposition \ref{060} establishes a relation between orthogonality of Schur functions in superspace $\{s_\Lambda\}_\Lambda$ and orthogonality of Schur functions in noncommuting variables in superespace $\{S_\Lambda\}_\Lambda$. Hence, we conclude that the functions $S_\Lambda$ are not orthogonal. For instance, due to Proposition \ref{059} and Proposition \ref{015}, we have\[S_{(0;1)}=m_{(\{0\},\{1\})}=2h_{(\{0\},\{1\})}-h_{(\{0,1\})};\qquad S_{(1;0)}=m_{(\{0,1\})}.\]Therefore, $\langle S_{(0;1)},S_{(1;0)}\rangle=-1$. In this sense, Schur function $S_\Lambda$ cannot be characterized by a triangular expansion on the monomials and an orthogonality property.

In the following, as in the case of commuting variables in superspace, the sets $\{S_\Lambda\}_\Lambda$ and $\{\bar{S}_\Omega\}_\Omega$ are essentially dual each other. Observe that if $I$ is a set superpartition with $\Lambda(I)=\Lambda$, then $(-1)^I=(-1)^{\deg(\Lambda)-\ell(\Lambda^s)}$. As $\omega(p_I)=(-1)^Ip_I$, by applying \eqref{076} we get that $\omega\circ\tilde{\rho}=\tilde{\rho}\circ\hat{\omega}$.

\begin{pro}\label{075}
For every superpartitions $\Lambda,\Omega\in\sPar_{n,m}$, we have:\[\ipro{\omega(\bar{S}_{\Lambda'}),S_\Omega}=\ipro{\omega(S_{\Lambda'}),\bar{S}_\Omega}=n!^2(-1)^{\binom{m}{2}}\delta_{\Lambda,\Omega}.\]
\end{pro}
\begin{proof}
By applying Proposition \ref{060}, Proposition \ref{063} and \cite[Corollary 29]{BlDeLaMa12}, we get
\[\ipro{\omega(\bar{S}_{\Lambda'}),S_\Omega}=\ipro{\omega(\tilde{\rho}(n!\bar{s}_{\Lambda'})),\tilde{\rho}(n!s_\Omega)}
=\ipro{\tilde{\rho}(\hat{\omega}(n!\bar{s}_{\Lambda'})),\tilde{\rho}(n!s_\Omega)}
=n!^2\ipro{\hat{\omega}(s_{\Lambda'}),s_\Omega}=n!^2(-1)^{\binom{m}{2}}\delta_{\Lambda,\Omega}.\]The second equality is shown similarly.
\end{proof}

\begin{rem}
Some particular elementary and homogeneous symmetric functions can be recovered from Schur functions by taking $\Lambda=(0;1^n)$ and $\Lambda=(n;\,)$, respectively. Specifically,\[S_{(0;1^n)}=e_{(\{0,1,\ldots,n\})};\qquad\omega(\bar{S}_{(n;\,)})=n!h_{(\{0,1,\ldots,n\})}.\]Indeed, we have $K_{(0;1^n),\Omega}=1$ whenever $\Omega=(0;1^n)$ and $K_{(0;1^n),\Omega}=0$ otherwise. Thus, by Proposition \ref{059}, we get\[S_{(0;1^n)}=\sum_{\Lambda(I)=(0;1^n)}m_{I}=m_{(\{0\},\{1\},\ldots,\{n\})}=e_{(\{0,1,\ldots,n\})},\]where the last equality follows from definition of $e_\Lambda$. On the other hand, by using Proposition \ref{075} and the fact that $\langle e_{(\{0,1,\ldots,n\})},h_J\rangle=\langle m_{(\{0\},\{1\},\ldots,\{n\})},h_J\rangle=(-1)^{\binom{m}{2}}n!\delta_{I,J}$, we obtain $\omega(\bar{S}_{(n;\,)})=n!h_{(\{0,1,\ldots,n\})}$.
\end{rem}

\begin{rem}\label{079}
Observe that $\{S_\Lambda\}_\Lambda$ and $\{\bar{S}_\Lambda\}_\Lambda$ are not bases for $\sNCSym$. Indeed, bases of $\sNCSym$ are indexed by set superpartitions, however there is no bijection between superpartitions and set superpartitions. In the non-superspace case, a first definition of Schur functions was given in \cite{RoSa06} by using classic semistandard Young tableaux, nevertheless these ones do not constitute a basis of $\NCSym$ either. A basis of Schur--type functions for $\NCSym$ was given in \cite{AlXiWi22} by using a noncommuting version of the Jacobi--Trudi determinant. They also proved that Schur functions in \cite{RoSa06} are sums of them. In superspace, even in the commuting case, there is no a version of the Jacobi--Trudi determinant. Thus, the Jacobi--Trudi identity for Schur functions is still an open problem, as does the case for our Schur--type functions.
\end{rem}

\begin{rem}
Recall that the product of classical Schur functions can be expressed as a linear combination of Schur functions by using the so--called Littlewood-Richardson coefficients. See for example \cite[Equation (7.64)]{St99}. In superspace, as $\{s_\Lambda\}_\Lambda$ is a basis of $\sSym$, every product $s_\Lambda s_\Omega$ can also be written as a linear combination of them, however the existence of a general combinatorial rule to get the coefficients involved here is still an open problem \cite[Section 9]{JoLa17}. As with classical Schur functions, in some particular cases, the product can be described by means of Pieri rules in superspace \cite[Section 4]{JoLa17}. In our case, as mentioned in Remark \ref{079}, the collection $\{S_\Lambda\}_\Lambda$ is not a basis for $\sNCSym$, and therefore the product $S_\Lambda S_\Omega$ is not necessarily a linear combination of these functions. For instance, by applying Proposition \ref{059}, we have\[S_{(0;1)}S_{(0;1)}=m_{(\{0\},\{1\})}m_{(\{0\},\{1\})}=-m_{(\{0,1\},\{0,2\})}.\]As mentioned in Subsection \ref{080}, the type of the set superpartition $(\{0,1\},\{0,2\})$ is not well defined, therefore this product is not a linear combination of our Schur functions. However, it could be an interesting problem to get Pieri--like rules for more specific cases, such as when the first superpartition is simply a partition.
\end{rem}

\begin{rem}
Given a superpartition $\Lambda$, we define\[{\sf m}_\Lambda=\Lambda!\sum_{\Lambda(I)=\Lambda}(-1)^{\varepsilon(I)}m_I.\]Observe that each $S_\Lambda$ is a linear combination of these functions (Proposition \ref{059}). It is not difficult to show that $\{{\sf m}_\Lambda\}_\Lambda$ is a linearly independent set. Thus, by applying the Gram--Schmidt process, with respect to the inner product in Subsection \ref{078}, we can get an orthonormal basis $\{{\sf S}_\Lambda\}_\Lambda$ for the vector subspace spanned by $\{{\sf m}_\Lambda\}_\Lambda$, which expand triangularly on them. Observe that, due to the orthogonality, this set does not coincide with either of the Schur functions introduced here. On the other hand, as explained above, by using Macdonald polynomials in superspace with parameters $q=t=1$, it was defined another family of Schur--type functions ${\sf s}_\Lambda$, which constitutes an orthogonal basis of $\sSym$ \cite[Equation (174)]{BlDeLaMa12}. As there is no a known method to obtain these functions by means of some tableaux--type objects, there is no a known combinatorial formula, as the ones given in \eqref{077}, to write the functions ${\sf s}_\Lambda$ in terms of the functions ${\sf m}_\Lambda$. Nevertheless, as the lifting map is an isometry (Proposition \ref{063}), we conjecture that $\tilde{\rho}({\sf s}_\Lambda)$ coincides with ${\sf S}_\Lambda$ up to some rational scalar. 
\end{rem}

\subsubsection*{Acknowledgements}

The first author acknowledges the financial support of DIDULS/ULS, through the project PR232146. The second author was partially supported by the grant ANID-FONDECYT Iniciación, No. 11241418.

%\label{081}

\bibliographystyle{plainurl}
\bibliography{../../../../Documents/Projects/LaTeX/bibtex.bib}

\end{document}